\documentclass[preprint,numbers,addressatend]{imsart}

\usepackage{amsthm,amsmath,natbib}
\usepackage{amsfonts,amssymb,graphicx}
\RequirePackage{hyperref}
\usepackage{color}

\RequirePackage{hypernat}

\arxiv{}

\startlocaldefs
\numberwithin{equation}{section}
\newtheorem{theo}{Theorem}[section]

\newtheorem{prop}[theo]{Proposition}
\newtheorem{lemma}[theo]{Lemma}
\newtheorem{defn}[theo]{Definition}
\newtheorem{rem}[theo]{Remark}

\newtheorem{cond}[theo]{Condition}
\newcommand{\var}{{\rm Var} \mspace{1mu}}

\renewcommand{\labelenumi}{\alph{enumi})}

\def\var{\mathop{\rm Var}\nolimits}

\usepackage{graphicx}
\usepackage{verbatim}

\endlocaldefs

\begin{document}

\begin{frontmatter}

\title{Directed Random Graphs with  \\ Given Degree Distributions}
\runtitle{Directed Random Graphs}


\author{\fnms{Ningyuan} \snm{Chen}\ead[label=e1]{nc2462@columbia.edu}}
\address{Department of Industrial Engineering \\ and Operations Research  \\ Columbia University \\ New York, NY 10027 \\ \printead{e1}}
\affiliation{Columbia University}
\and \hspace{4pt}
\author{\fnms{Mariana} \snm{Olvera-Cravioto}\ead[label=e2]{molvera@ieor.columbia.edu}}
\address{Department of Industrial Engineering \\ and Operations Research \\ Columbia University \\ New York, NY 10027 \\ \printead{e2}}
\affiliation{Columbia University}

\runauthor{N. Chen  and M. Olvera-Cravioto}

\begin{abstract}
Given two distributions $F$ and $G$ on the nonnegative integers we propose an algorithm to construct in- and out-degree sequences from samples of i.i.d. observations from $F$ and $G$, respectively, that with high probability will be graphical, that is, from which a simple directed graph can be drawn. We then analyze a directed version of the configuration model and show that, provided that $F$ and $G$ have finite variance, the probability of obtaining a simple graph is bounded away from zero as the number of nodes grows. We show that conditional on the resulting graph being simple, the in- and out-degree distributions are (approximately) $F$ and $G$ for large size graphs. Moreover, when the degree distributions have only finite mean we show that the elimination of self-loops and multiple edges does not significantly change the degree distributions in the resulting simple graph. 
\end{abstract}



\begin{date}
\date{\today}
\end{date}

\end{frontmatter}

\section{Introduction}

In order to study complex systems such as the World Wide Web (WWW) we propose a model for generating a simple directed random graph with prescribed degree distributions. The ability to match degree distributions to real graphs is perhaps the first characteristic one would desire from a model, and although several models that accomplish this for undirected graphs have been proposed in the recent literature \cite{McKay_Worm_90a, Chung_Lu_02a, Chung_Lu_02b, Britt_Deij_Mart_06}, not much has been done for the directed case.  In the WWW example that motivates this work, vertices represent webpages and the edges represent the links between them. Empirical studies (e.g., \cite{Kleinberg_etal_99, Broder_etal_00}) suggest that both the {\em in-degree} and {\em out-degree}, number of links pointing to a page and the number of outbound links of a page, respectively, follow a power-law distribution, a characteristic often referred to as the {\em scale-free} property. 

The model we propose in this paper is closely related to the work in \cite{Britt_Deij_Mart_06} for undirected graphs, where given a probability distribution $F$, the goal is to provide an algorithm to generate a simple random graph whose degree distribution is approximately $F$. Two of the models presented in \cite{Britt_Deij_Mart_06}, as well as the model in \cite{VDH_etal_05}, are in turn related to the well-known configuration model \cite{Wormald_78, Bollobas_80}, where vertices are given stubs or half-edges according to a degree sequence $\{d_i\}$ and these stubs are then randomly paired to form edges. To obtain a prescribed degree distribution, the degree sequence $\{d_i\}$ is chosen as i.i.d. random variables having distribution $F$. This method allows great flexibility in terms of the generality of $F$, which is very important in the applications we have in mind. The most general of the results presented here require only that the degree distributions have finite $(1+\epsilon)$th  moment, and are therefore applicable to a great variety of examples, including the WWW.

For a directed random graph there are two distributions that need to be chosen, the in-degree and out-degree distributions, denoted respectively $F = \{ f_k: k \geq 0\}$ and $G = \{g_k: k \geq 0\}$. The in-degree of a node corresponds to the number of edges pointing to it, while the out-degree is the number of edges pointing out. To follow the ideas from \cite{VDH_etal_05,Britt_Deij_Mart_06}, we propose to draw the in-degree and out-degree sequences as i.i.d. observations from distributions $F$ and $G$. Unlike the undirected case where the only main problem with this approach is that the sum of the degrees might not be even, which is necessary to draw an undirected graph, in the directed case the corresponding condition is that the sum of the in-degrees and the sum of the out-degrees be the same. Since the probability that two i.i.d. sequences will have the same sum, even if their means are equal, converges to zero as the number of nodes grows to infinity, the first part of the paper focuses on how to construct valid degree sequences without significantly destroying their i.i.d. properties. Once we have valid degree sequences the problem is how to obtain a simple graph, since the random pairing may produce self-loops and multiple edges in the same direction. This problem is addressed in two ways, the first of which consists in showing sufficient conditions under which the probability of generating a simple graph through random pairing is strictly positive, which in turn suggests repeating the pairing process until a simple graph is obtained. The second approach is to simply erase the self-loops and multiple edges of the resulting graph. In both cases, one must show that the  degree distributions in the final simple graph remain essentially unchanged. In particular, if we let $f_k^{(n)}$ be the probability that a randomly chosen node from a graph of size $n$ has in-degree $k$, and let $g_k^{(n)}$ be the corresponding probability for the out-degree, then we will show that,  
$$f_k^{(n)} \to f_k \qquad \text{and} \qquad  g_k^{(n)} \to g_k,$$
as $n \to \infty$. We also prove a similar result for the empirical distributions. 

The question of whether a given pair of in- and out-degree sequences $(\{m_i\}, \{d_i\})$ is graphical, i.e., from which it is possible to draw a simple directed graph, has been recently studied in \cite{Erd_Mik_Tor_10, LaMar_10}, where algorithms to realize such graphs have also been analyzed. Random directed graphs with arbitrary degree distributions have been studied in \cite{New_Str_Wat_01} via generating functions, which can be used to formalize concepts such as ``in-components" and ``out-components" as well as to estimate their average size.  Models of growing networks that can be calibrated to mimic the power-law behavior of the WWW have been analyzed using statistical physics techniques in \cite{Kra_Rod_Red_01, Kra_Red_02}. The approach followed in this paper focuses on one hand on the generation of in- and out-degree sequences that are close to being i.i.d. and that are graphical with high probability, and on the other hand on providing conditions under which a simple graph can be obtained through random pairing. The directed configuration model with (close to) i.i.d. degree sequences, although not a growing network model, has the advantage of being analytically tractable and easy to simulate.

The rest of the paper is organized as follows. In Section \ref{S.DegreeSequences} we introduce a model to construct in- and out-degree sequences that are very close to being two independent sequences of  i.i.d. random variables having distributions $F$ and $G$, respectively, but whose sums are the same; in the same spirit as the results in \cite{Arr_Ligg_05} we also show that the suggested method produces with high probability a graphical pair of degree sequences.  In Subsection \ref{SS.Repeat} we prove sufficient conditions under which the probability that the directed configuration model will produce a simple graph will be bounded away from zero, and show that conditional on the resulting graph being simple, the degree sequences have asymptotically the correct distributions. In Subsection \ref{SS.Erase} we show that under very mild conditions, the process of simply erasing self-loops and multiple edges results in a graph whose degree distributions are still asymptotically $F$ and $G$.

\section{Graphs and degree sequences}
\label{S.DegreeSequences}

As mentioned in the introduction, the goal of this paper is to provide an algorithm for generating a random directed graph with $n$ nodes with the property that its in-degrees and out-degrees have some prespecified distributions $F$ and $G$, respectively. Moreover, we would like the resulting graph to be {\em simple}, that is, it should not contain self-loops or multiple edges in the same direction. The two models that we propose are based on the so-called configuration or pairing model, which produces a random undirected graph from a degree sequence $\{d_1, d_2, \dots, d_n\}$.  In \cite{VDH_etal_05, Britt_Deij_Mart_06} the prescribed degree distribution is obtained by drawing the degree sequence $\{d_i\}$ as i.i.d. random variables from that distribution. More details about the configuration model can be found in Section \ref{S.ConfigurationModel}.

Following the same idea of using a sequence of  i.i.d. random variables to generate the degree sequence of an undirected graph, the natural extension to the directed case would be to draw two i.i.d. sequences from given distributions $F$ and $G$. We note that in the undirected setting the two main problems with this approach are: 1) that the sum of the degrees may be odd, in which case it is impossible to draw a graph, and 2) that there may not exist a simple graph having the prescribed degree sequence. The first problem is easily fixed by either sampling the i.i.d. sequence until its sum is even (which will happen with probability 1/2 asymptotically), or simply adding one to the last random number in the sequence. The second problem, although related to the verification of graphicality criteria (e.g., the Erd\"os-Gallai criterion \cite{Erd_Gall_60}), turns out to be negligible as the number of nodes goes to infinity, as the work in \cite{Arr_Ligg_05} shows. For directed graphs a graphicality criterion also exists, and the second problem turns out to be negligible for large graphs just as in the undirected case. Nonetheless, the equivalent of the first problem is now that the potential in-degree and out-degree sequences must have the same sum, which is considerably harder to fix. Before proceeding with the formulation of our proposed algorithm we give some basic definitions which will be used throughout the paper.

\begin{defn}
We denote by $\vec{G}(V, \vec{E})$ a directed graph on $n$ nodes or vertices, $V = \{v_1, v_2, \dots, v_n\}$, connected via the set of directed edges $\vec{E}$. 
\end{defn}

\begin{defn}
We say that  $\vec{G}(V, \vec{E})$ is {\em simple} if any pair of nodes are connected by at most one edge in each direction, and if there are no edges in between a node and itself.  
\end{defn}

\begin{defn}
The in-degree $m_i$, respectively, out-degree $d_i$, of node $v_i \in V$ is the total number of edges from other nodes to $v_i$, respectively, from $v_i$ to other nodes.  The pair of sequences $({\bf m}, {\bf d}) = (\{ m_1, m_2, \dots, m_n \}, \linebreak \{ d_1, d_2, \dots, d_n\} )$ of nonnegative integers is called a {\em bi-degree-sequence} if $m_i$ and $d_i$ correspond to the in-degree and out-degree, respectively, of node $v_i$. 
\end{defn}

\begin{defn} \label{D.bi-degree-sequence}
A bi-degree-sequence $({\bf m}, {\bf d})$ is said to be {\em graphical} if there exists a simple directed graph $\vec{G}(V, \vec{E})$ on the set of nodes $V$ such that the in-degree and out-degree sequences together form $({\bf m}, {\bf d})$. In this case we say that $\vec{G}$ realizes the bi-degree-sequence. 
\end{defn}

In view of these definitions our goal is to generate the sequences $\{ m_i \}$ and $\{ d_i \}$ from i.i.d. samples of given distributions $F = \{ f_k: k \geq 0\}$ and $G = \{ g_k: k \geq 0\}$, respectively. Both $F$ and $G$ are assumed to be probability distributions with support on the nonnegative integers with a finite common mean $\mu$. Note that although the Strong Law of Large Numbers (SLLN) guarantees that if we simply sample i.i.d. random variables $\{\gamma_1, \dots, \gamma_n\}$ from $F$ and, independently, i.i.d. random variables $\{\xi_1, \dots, \xi_n\}$ from $G$, then
$$P\left(\lim_{n \to \infty} \frac{1}{n} \sum_{i=1}^n \gamma_i = \lim_{n\to\infty} \frac{1}{n} \sum_{i=1}^n \xi_i \right) = 1,$$
it is also true that in general 
$$\lim_{n \to \infty} P\left(   \sum_{i=1}^n \gamma_i - \sum_{i=1}^n \xi_i  = 0  \right) = 0.$$
One potential idea to fix the problem is to sample one of the two sequences, say the in-degrees, as i.i.d. observations $\{\gamma_1, \dots, \gamma_n\}$ from $F$ and then sample the second sequence from the conditional distribution $G$ given that its sum is $\Gamma_n = \sum_{i=1}^n \gamma_i$. This approach has the major drawback that this conditional distribution may be ill-behaved, in the sense that the probability of the conditioning event, the sum being equal to $\Gamma_n$, converges to zero in most cases. It follows that we need a different mechanism to sample the degree sequences. The precise algorithm we propose is described below; we focus on first sampling two independent i.i.d. sequences and then add in- or out-degrees as needed to match their sums.

The following definition will be needed throughout the rest of the paper.

\begin{defn}
We say that a function $L(\cdot)$ is slowly varying at infinity if $\lim_{x \to \infty} L(tx)/L(x) = 1$ for all fixed $t > 0$. A distribution function $F$ is said to be regularly varying with index $\alpha > 0$, $F \in \mathcal{R}_{-\alpha}$, if $\overline{F}(x) = 1 - F(x) = x^{-\alpha} L(x)$ with $L(\cdot)$ slowly varying. 
\end{defn}

We will also use the notation $\Rightarrow$ to denote convergence in distribution, $\stackrel{P}{\longrightarrow}$ to denote convergence in probability, and $\mathbb{N} = \{1, 2, 3, \dots\}$ to refer to the positive integers.

\subsection{The Algorithm} \label{SS.Algorithm}

We assume that the target degree distributions $F$ and $G$ have support on the nonnegative integers and have common mean $\mu > 0$. Moreover, suppose that there exist slowly varying functions $L_F(\cdot)$ and $L_G(\cdot)$ such that
\begin{equation} \label{eq:powerlaw-majorant}
\overline{F}(x) = \sum_{k > x} f_k \leq x^{-\alpha} L_F(x) \qquad \text{and} \qquad \overline{G}(x) = \sum_{k > x} g_k \leq x^{-\beta} L_G(x),
\end{equation}
for all $x \geq 0$, where $\alpha, \beta > 1$. 

We refer the reader to \cite{BiGoTe_1987} for all the properties of slowly varying functions that will be used in the proofs. However, we do point out here that the tail conditions in \eqref{eq:powerlaw-majorant} ensure that $F$ has finite moments of order $s$ for all $0 < s < \alpha$, and $G$ has finite moments of order $s$ for all $0 < s < \beta$.  The constant 
$$\kappa = \min\{ 1 - \alpha^{-1}, 1- \beta^{-1}, 1/2\},$$
will play an important role throughout the paper. The algorithm is given below. 

\begin{enumerate}
\renewcommand{\labelenumi}{\arabic{enumi}.}
\item Fix $0 < \delta_0 < \kappa$. 
\item Sample an i.i.d. sequence $\{\gamma_1, \dots, \gamma_n\}$ from distribution $F$; let $\Gamma_n=\sum_{i=1}^n\eta_i$.

\item Sample an i.i.d. sequence $\{ \xi_1, \dots, \xi_n\}$ from distribution $G$; let $\Xi_n=\sum_{i=1}^n\xi_i$.

\item Define $\Delta_n= \Gamma_n- \Xi_n$. If $|\Delta_n| \leq n^{1-\kappa + \delta_0}$ proceed to step 5; otherwise repeat from step 2. 
 
\item Choose randomly $|\Delta_n|$ nodes $\{i_1, i_2, \dots, i_{|\Delta_n|}\}$ without remplacement and let
$$N_i = \eta_i + \tau_i, \qquad D_i = \xi_i + \chi_i, \qquad i = 1, 2, \dots, n,$$
where
\begin{align*}
\chi_i &= \begin{cases}
			1 & \text{if $\Delta_n \geq 0$ and $i\in \{ i_1,i_2,\dots,i_{\Delta_n} \} $,}\\
			0 & \text{otherwise,}
	\end{cases} \qquad \text{and} \\
\tau_i &= \begin{cases}
			1 & \text{if $\Delta_n < 0$ and $i\in \{ i_1,i_2,\dots,i_{|\Delta_n|} \} $,}\\
			0 & \text{otherwise.}
	\end{cases}
\end{align*}
\end{enumerate}

\begin{rem}
(i) This algorithm constructs a bi-degree-sequence $({\bf N}, {\bf D})$ having the property that $L_n = \sum_{i=1}^n N_i = \sum_{i=1}^n D_i$. 
(ii) Note that we have used the capital letters $N_i$ and $D_i$ to denote the in-degree and out-degree, respectively, of node $i$, as opposed to using the notation $m_i$ and $d_i$ from Definition \ref{D.bi-degree-sequence}; we do this to emphasize the randomness of the bi-degree-sequence itself. 
(iii) Clearly, neither $\{N_1, \dots, N_n\}$ nor $\{D_1, \dots, D_n\}$ are i.i.d. sequences, nor are they independent of each other, but we will show in the next section that asymptotically as $n$ grows to infinity they have the same joint distribution as $(\{ \gamma_i\}, \{\xi_i\})$. (iv) We will also show that the condition in step 4 has probability converging to one. (v) Note that we always choose to add degrees, rather than fixing one sequence and always adjust the other one, to avoid having problems with nodes with in- or out-degree zero. 
\end{rem}

\subsection{Asymptotic behavior of the degree sequence}

We now provide some results about the asymptotic behavior of the bi-degree-sequence obtained from the algorithm we propose. The first thing we need to prove is that the algorithm will always end in finite time, and the only step where we need to be careful is in Step 4, since it may not be obvious that we can always draw two independent i.i.d. sequences satisfying $|\Delta_n| \leq n^{1-\kappa+\delta_0}$ in a reasonable amount of time. The first lemma we give establishes that this is indeed the case by showing  that the probability of satisfying condition $|\Delta_n| \leq n^{1-\kappa+\delta_0}$ converges to one as the size of the graph grows. All the proofs in this section can be found in Subsection \ref{SS.DegreeSequencesProofs}. 

\begin{lemma} \label{L.DifferenceLimit}
Define $\mathcal{D}_n = \{ |\Delta_n| \leq n^{1-\kappa+\delta_0} \}$, then
$$\lim_{n \to \infty} P(\mathcal{D}_n ) = 1.$$
\end{lemma}

Since the sums of the in-degrees and out-degrees are the same, we can always draw a graph, but this is not enough to guarantee that we can draw a simple graph.  In other words, we need to determine with what probability will the the bi-degree-sequence $({\bf N}, {\bf D})$ be graphical, and to do this we first need a appropriate criterion, e.g., a directed version of the Erd\"os-Gallai criterion for undirected graphs. The following result (Corollary 1 on p. 110 in \cite{Berge_1976}) gives necessary and sufficient conditions for a bi-degree-sequence to be graphical; the original statement is for more general $p$-graphs, where up to $p$ parallel edges in the same direction are allowed. The notation $|A|$ denotes the cardinality of set $A$. 

\begin{theo} \label{thm:graphical-condition} 
Given a set of $n$ vertices $V= \{v_{1}, \dots, v_{n}\}$, having bi-degree-sequence $({\bf m}, {\bf d}) = \linebreak (\{m_1, \dots, m_n\}, \{d_1, \dots, d_n\})$, a necessary and sufficient condition for $({\bf m}, {\bf d})$ to be graphical is 
\begin{enumerate}
\item $\displaystyle \sum_{i=1}^n m_i = \sum_{i=1}^n d_i$, and
\item $\displaystyle \sum_{i=1}^{n} \min\{d_{i}, |A-\{v_{i}\}|\} \geq \sum_{v_{i}\in A} m_{i} $ for any $A \subseteq V$.
\end{enumerate}
\end{theo}

We now state a result that shows that for large $n$, the bi-degree-sequence $({\bf N}, {\bf D})$ constructed in Subsection \ref{SS.Algorithm} is with high probability graphical. Related results for undirected graphs can be found in \cite{Arr_Ligg_05}, which includes the case when the degree distribution has infinite mean. 

\begin{theo}\label{thm:always-graphical}
For the bi-degree-sequence $({\bf N}, {\bf D})$ constructed in Section \ref{SS.Algorithm} we have 
$$\lim_{n \to \infty} P\left( ({\bf N}, {\bf D}) \mbox{ is graphical}\right) = 1.$$
\end{theo}

The second property of  $({\bf N}, {\bf D})$ that we want to show is that despite the fact that the sequences $\{N_i\}$ and $\{D_i\}$ are no longer independent nor individually i.i.d., they are still asymptotically so as the number of vertices $n$ goes to infinity.  The intuition behind this result is that the number of degrees that need to be added to one of  the i.i.d. sequences $\{ \gamma_i\}$ or $\{\xi_i\}$ to match their sum is small compared to $n$, and therefore the sequences $\{N_i\}$ and $\{D_i\}$ are {\em almost} i.i.d. and independent of each other. This feature makes the bi-degree-sequence $({\bf N}, {\bf D})$ we propose an approximate equivalent of the i.i.d. degree sequence considered in \cite{Arr_Ligg_05, VDH_etal_05, Britt_Deij_Mart_06} for undirected graphs.

\begin{theo} \label{T.WeakConvergence}
The bi-degree-sequence $({\bf N}, {\bf D})$ constructed in Subsection \ref{SS.Algorithm} satisfies that for any fixed $r, s \in \mathbb{N}$, 
$$( N_{i_1}, \dots, N_{i_r}, D_{j_1}, \dots, D_{j_s} ) \Rightarrow (\gamma_1, \dots, \gamma_r, \xi_1, \dots, \xi_s)$$
as $n \to \infty$, where $\{\gamma_i\}$ and $\{\xi_i\}$ are independent sequences of i.i.d. random variables having distributions $F$ and $G$, respectively. 
\end{theo}

To end this section, we give a result that establishes regularity conditions of the bi-degree-sequence $({\bf N}, {\bf D})$ which will be important in the sequel.

\begin{prop} \label{P.MainConditions}
The bi-degree-sequence $({\bf N}, {\bf D})$ constructed in Subsection \ref{SS.Algorithm} satisfies
$$\frac{1}{n} \sum_{k=1}^n 1(N_k = i, D_k = j) \stackrel{P}{\longrightarrow} f_i g_j,  \quad \text{for all } i,j \in \mathbb{N} \cup \{0\},$$
$$\frac{1}{n} \sum_{i=1}^n N_i \stackrel{P}{\longrightarrow} E[\gamma_1], \qquad \frac{1}{n} \sum_{i=1}^n D_i \stackrel{P}{\longrightarrow} E[\xi_1], \qquad \text{and} \qquad \frac{1}{n} \sum_{i=1}^n N_iD_i \stackrel{P}{\longrightarrow} E[\gamma_1 \xi_1],$$
as $n \to \infty$, and provided $E[\gamma_1^2 + \xi_1^2] < \infty$, 
$$\frac{1}{n} \sum_{i=1}^n N_i^2 \stackrel{P}{\longrightarrow} E[\gamma_1^2], \qquad \text{and} \qquad \frac{1}{n} \sum_{i=1}^n D_i^2 \stackrel{P}{\longrightarrow} E[\xi_1^2],$$
as $n \to \infty$. 
\end{prop}


\section{The configuration model} \label{S.ConfigurationModel}

In the previous section we introduced a model for the generation of a bi-degree-sequence $({\bf N}, {\bf D})$ that is close to being a pair of independent sequences of i.i.d. random variables, but yet has the property of being graphical with probability close to one as the size of the graph goes to infinity. We now turn our attention to the problem of obtaining a realization of such sequence, in particular, of drawing a simple graph having $({\bf N}, {\bf D})$ as its bi-degree-sequence. 

The approach that we follow is a directed version of the {\em configuration model}. The configuration, or {\em pairing model}, was introduced in \cite{Bollobas_80} and \cite{Wormald_78}, although earlier related ideas based on symmetric matrices with $\{0, 1\}$ entries go back to the early 70's; see \cite{Wormald_99, Bollobas_2001} for a survey of the history as well as additional references. The configuration model is based on the following idea: given a degree sequence ${\bf d} = \{d_1, \dots, d_n\}$, to each node $v_i$, $1 \leq i \leq n$, assign $d_i$ stubs or half-edges, and then pair half-edges to form an edge in the graph by randomly selecting with equal probability from the remaining set of unpaired half-edges. This procedure results in a multigraph on $n$ nodes having ${\bf d}$ as its degree sequence, where the term multigraph refers to the possibility of self-loops and multiple edges. Although this algorithm does not produce a multigraph uniformly chosen at random from the set of all multigraphs having degree sequence ${\bf d}$, a simple graph uniformly chosen at random can be obtained by choosing a pairing uniformly at random and discarding the outcome if it has self-loops or multiple edges \cite{Wormald_99}.  The question that becomes important then is to estimate the probability with which the pairing model will produce a simple graph. For the undirected graph setting we have described, such results were given in \cite{Read_60, Ben_Can_78, Wormald_78, Bollobas_80, McKay_Worm_91} for regular $d$-graphs (graphs where each node has exactly degree $d$), and in \cite{McKay_Worm_90b, McKay_Worm_91,van2009random} for general graphical degree sequences. 

From the previous discussion, it should be clear that it is important to determine conditions under which the probability of obtaining a simple graph in the pairing model is bounded away from zero as $n \to \infty$.  Such conditions are essentially bounds on the rate of growth of the maximum (minimum) degree and/or the existence of certain limits (see, e.g., \cite{McKay_Worm_90b, McKay_Worm_91, van2009random}). The set of conditions given below is taken from \cite{van2009random}, and we include it here as a reference for the directed version discussed in this paper.

\begin{cond} \label{Cond.RegularityUndirected}
Given a degree sequence ${\bf d} = \{d_1, \dots, d_n\}$, let $D^{[n]}$ be the degree of a randomly chosen node, i.e.,
$$P(D^{[n]} = k) = \frac{1}{n} \sum_{i=1}^n 1(d_i = k).$$
\begin{enumerate}
\item {\em Weak convergence.} There exists a finite random variable $D$ taking values on the positive integers such that
$$D^{[n]} \Rightarrow D, \qquad n \to \infty.$$

\item {\em Convergence of the first moment.} 
$$\lim_{n\to\infty} E[D^{[n]}] = E[D].$$

\item {\em Convergence of the second moment.}
$$\lim_{n\to\infty} E[ (D^{[n]})^2] = E[D^2].$$
\end{enumerate}
\end{cond} 

\begin{rem}
It is straightforward to verify that if the degree sequence is chosen as an i.i.d. sample $\{D_1, \dots, D_n\}$ from some distribution $F$ on the positive integers having finite first moment, then parts (a) and (b) of Condition \ref{Cond.RegularityUndirected} are satisfied, and if $F$ has finite second moment then also part (c) is satisfied; the adjustment made to ensure that the sum of the degrees is even, if needed, can be shown to be negligible.  
\end{rem}

Condition \ref{Cond.RegularityUndirected} guarantees that the probability of obtaining a simple graph in the pairing model is bounded away from zero (see, e.g., \cite{van2009random}), in which case we can obtain a uniformly simple realization of the (graphical) degree sequence $\{d_i\}$ by repeating the random pairing until a simple graph is obtained. When part (c) of Condition  \ref{Cond.RegularityUndirected} fails, then an alternative is to simply erase the self-loops and multiple edges. These two approaches give rise to the {\em repeated} an {\em erased} configuration models, respectively.

Having given a brief description of the configuration model for undirected graphs, we will now discuss how to adapt it to draw directed graphs. The idea is basically the same, given a bi-degree-sequence $({\bf m}, {\bf d})$, to each node $v_i$ assign $m_i$ inbound half-edges and $d_i$ outbound half-edges; then, proceed to match inbound half-edges to outbound half-edges to form directed edges. To be more precise, for each unpaired inbound half-edge of node $v_i$ choose randomly from all the available unpaired outbound half-edges, and if the selected outbound half-edge belongs to node, say, $v_j$, then add a directed edge from $v_j$ to $v_i$ to the graph; proceed in this way until all unpaired inbound half-edges are matched. The following result shows that conditional on the graph being simple, it is uniformly chosen among all simple directed graphs having bi-degree-sequence $({\bf m}, {\bf d})$. All the proofs of Section \ref{S.ConfigurationModel} can be found in Subsection \ref{SS.ConfigurationModelProofs}. 

\begin{prop} \label{P.Uniformity}
Given a graphical bi-degree-sequence $({\bf m}, {\bf d})$, generate a directed graph according to the directed configuration model. Then, conditional on the obtained graph being simple, it is uniformly distributed among all simple directed graphs having bi-degree-sequence $({\bf m}, {\bf d})$.
\end{prop}

The question is now under what conditions will the probability of obtaining a simple graph be bounded away from zero as the number of nodes, $n$, goes to infinity. When this probability is bounded away from zero we can repeat the random pairing until we draw a simple graph: the repeated model; otherwise, we can always erase the self-loops and multiple edges in the same direction to obtain a simple graph: the erased model. These two models are discussed in more detail in the following two subsections, where we also provide sufficient conditions under which the the probability of obtaining a simple graph will be bounded away from zero.  

We end this section by mentioning that another important line of problems related to the drawing of simple graphs (directed or undirected) is the development of efficient simulation algorithms, see for example the recent work in \cite{Blit_Diac_11} using importance sampling techniques for drawing a simple graph with prescribed degree sequence $\{d_i\}$; similar ideas should also be applicable to the directed model.

\subsection{Repeated Directed Configuration Model}
\label{SS.Repeat}

In this section we analyze the directed configuration model using the bi-degree-sequence $({\bf N}, {\bf D})$ constructed in Subsection \ref{SS.Algorithm}. In order to do so we will first need to establish sufficient conditions under which the probability that the directed configuration model produces a simple graph is bounded away from zero as the number of nodes goes to infinity. Since this property does not directly depend on the specific bi-degree-sequence $({\bf N}, {\bf D})$, we will prove the result for general bi-degree-sequences $({\bf m}, {\bf d})$ satisfying an analogue of Condition \ref{Cond.RegularityUndirected}. As one may expect, we will require the existence of certain limits related to the (joint) distribution of the in-degree and out-degree of a randomly chosen node. Also, since the sequences $\{m_i\}$ and $\{d_i\}$ need to have the same sum, we prefer to consider a sequence of bi-degree-sequences, i.e., $\{({\bf m}_n, {\bf d}_n)\}_{n \in \mathbb{N}}$ where $({\bf m}_n, {\bf d}_n) = (\{ m_{n1}, \dots, m_{nn} \}, \{d_{n1}, \dots, d_{nn}\})$, since otherwise the equal sum constraint would greatly restrict the type of sequences we can use (e.g., $m_i = d_i$ for all $i \in \mathbb{N}$). The corresponding version of Condition~\ref{Cond.RegularityUndirected} is given below. 

\begin{cond} \label{Cond.RegularityDirected}
Given a sequence of bi-degree-sequences $\{({\bf m}_n, {\bf d}_n)\}_{n \in \mathbb{N}}$ satisfying 
$$\sum_{i=1}^n m_{ni} = \sum_{i=1}^n d_{ni} \qquad \text{for all $n$},$$
let $(N^{[n]}, D^{[n]})$ denote the in-degree and out-degree of a randomly chosen node, i.e.,
$$P((N^{[n]} , D^{[n]}) = (i,j)) = \frac{1}{n} \sum_{k=1}^n 1(m_{nk} = i, d_{nk} = j).$$
\begin{enumerate}
\item {\em Weak convergence.} There exist finite random variables $\gamma$ and $\xi$ taking values on the nonnegative integers and satisfying $E[\gamma] = E[\xi]  > 0$ such that
$$(N^{[n]} , D^{[n]}) \Rightarrow (\gamma, \xi), \qquad n \to \infty.$$

\item {\em Convergence of the first moments.} 
$$\lim_{n\to\infty} E[N^{[n]}] = E[\gamma] \qquad \text{and} \qquad \lim_{n\to\infty} E[D^{[n]}] = E[\xi].$$

\item {\em Convergence of the covariance.}
$$\lim_{n \to \infty} E[ N^{[n]} D^{[n]} ] = E[\gamma \xi]. $$

\item {\em Convergence of the second moments.}
$$\lim_{n\to\infty} E[ (N^{[n]})^2] = E[\gamma^2] \qquad \text{and} \qquad \lim_{n\to\infty} E[ (D^{[n]})^2] = E[\xi^2].$$

\end{enumerate}
\end{cond}

We now state a result that says that the number of self-loops and the number of multiple edges produced by the random pairing converge jointly, as $n \to \infty$, to a pair of independent Poisson random variables. As a corollary we obtain that the probability of the resulting graph being simple converges to a positive number, and is therefore bounded away from zero. The proof is an adaptation of the proof of Proposition 7.9 in \cite{van2009random}. 

Consider the multigraph obtained through the directed configuration model from the bi-degree-sequence $({\bf m}_n, {\bf d}_n)$, and let $S_n$ be the number of self-loops and $M_n$ be the number of multiple edges in the same direction, that is, if there are $k \geq 2$ (directed) edges from node $v_i$ to node $v_j$, they contribute $(k-1)$ to $M_n$. 

\begin{prop} \label{P.PoissonLimit}
(Poisson limit of self-loops and multiple edges) If $\{({\bf m}_n, {\bf d}_n)\}_{n \in \mathbb{N}}$ satisfies Condition~\ref{Cond.RegularityDirected} with $E[\gamma] = E[\xi] = \mu > 0$, then 
$$(S_n, M_n) \Rightarrow (S, M)$$
as $n \to \infty$, where $S$ and $M$ are two independent Poisson random variables with means 
$$\lambda_1 = \frac{E[\gamma \xi]}{\mu} \qquad \text{and} \qquad \lambda_2 = \frac{ E[\gamma(\gamma-1)] E[\xi(\xi-1)]}{2 \mu^2},$$
respectively. 
\end{prop}

Since the probability of the graph being simple is $P( S_n = 0, M_n = 0)$, we obtain as a consequence the following theorem.

\begin{theo} \label{T.ProbSimple}
Under the assumptions of Proposition \ref{P.PoissonLimit}, 
$$\lim_{n \to \infty} P(\text{\rm graph obtained from $({\bf m}_n, {\bf d}_n)$ is simple}) = e^{-\lambda_1 - \lambda_2} > 0.$$
\end{theo}

It is clear from Proposition \ref{P.MainConditions} that Condition \ref{Cond.RegularityDirected} is satisfied by the bi-degree-sequence $({\bf N}, {\bf D})$ proposed in Subsection \ref{SS.Algorithm} whenever $F$ and $G$ have finite variance. This implies that one way of obtaining a simple directed graph on $n$ nodes is by first sampling the bi-degree-sequence $({\bf N}, {\bf D})$ according to Subsection \ref{SS.Algorithm}, then checking if it is graphical, and if it is, use the directed pairing model to draw a graph, discarding any realizations that are not simple. Alternatively, since the probability of $({\bf N}, {\bf D})$ being graphical converges to one, then one could skip the verification of graphicality and re-sample $({\bf N}, {\bf D})$ each time the pairing needs to be  repeated. 

The last thing we show in this section is that the degree distributions of the resulting simple graph will have with high probability the prescribed degree distributions $F$ and $G$, as required. More specifically, if we let $({\bf N}^{(r)}, {\bf D}^{(r)})$ be the bi-degree-sequence of the final simple graph obtained through the repeated directed configuration model with bi-degree-sequence $({\bf N}, {\bf D})$, then we will show that the joint distribution
$$h^{(n)}(i,j) = \frac{1}{n} \sum_{k=1}^n P(N_k^{(r)} = i, D_k^{(r)} = j)   \qquad i,j = 0, 1, 2, \dots,$$
converges to $f_i g_j$, and the empirical distributions,  
$$\widehat{f_k}^{(n)} = \frac{1}{n} \sum_{i=1}^n 1(N_i^{(r)} = k) \qquad \text{and} \qquad \widehat{g_k}^{(n)} = \frac{1}{n} \sum_{i=1}^n 1(D_i^{(r)} = k) \qquad k = 0, 1, 2, \dots,$$
converge in probability to $f_k$ and $g_k$, respectively.  The same result was shown in \cite{Britt_Deij_Mart_06} for the undirected case with i.i.d. degree sequence $\{D_i\}$. 

\begin{prop} \label{P.RepeatedDistr}
For the repeated directed configuration model with bi-degree-sequence $({\bf N}, {\bf D})$, as constructed in Subsection \ref{SS.Algorithm}, we have: 
\begin{enumerate}
\item $h^{(n)}(i,j) \to f_i g_j$  as $n \to \infty$, $i,j = 0, 1, 2, \dots$, and

\item for all $k = 0, 1, 2, \dots$, 
$$\widehat{f_k}^{(n)} \stackrel{P}{\longrightarrow} f_k \qquad \text{and} \qquad\widehat{g_k}^{(n)} \stackrel{P}{\longrightarrow} g_k, \qquad n \to \infty.$$
\end{enumerate}
\end{prop}

\begin{rem}
Note that by the continuous mapping theorem, (a) implies that the marginal distributions of the in-degrees and out-degrees,
$$f^{(n)}(i) =  \frac{1}{n} \sum_{k=1}^n P(N_k^{(r)} =i) \qquad \text{and} \qquad g^{(n)}(j) = \frac{1}{n} \sum_{k=1}^n P(D_k^{(r)} =j),$$
converge to $f_i$ and $g_j$, respectively. The same arguments used in the proof also give that the joint empirical distribution converges to $f_i g_j$ in probability.
\end{rem}

\subsection{Erased directed configuration model}
\label{SS.Erase}

In this section we consider the erased directed configuration model, which is particularly useful when the probability of drawing a simple graph converges to zero as the number of nodes increases, which could happen, for example, when Condition~\ref{Cond.RegularityDirected} (d) fails. Given a bi-degree-sequence $({\bf m}, {\bf d})$, the erased model consists in first obtaining a multigraph according to the directed configuration model and then erase all self-loops and merge multiple edges in the same direction into a single edge, with the result being a simple graph. Note that the graph obtained through this process no longer has $({\bf m}, {\bf d})$ as its bi-degree-sequence.

As for the repeated model, let $({\bf N}^{(e)}, {\bf D}^{(e)})$ be the bi-degree-sequence of the simple graph obtained through the erased directed configuration model with bi-degree-sequence $({\bf N}, {\bf D})$. Define the joint distribution 
$$h^{(n)}(i,j) =  \frac{1}{n} \sum_{k=1}^n P(N_k^{(e)} = i, D_k^{(e)} = j)   \qquad i,j = 0, 1, 2, \dots,$$
and the empirical distributions,  
$$\widehat{f_k}^{(n)} = \frac{1}{n} \sum_{i=1}^n 1(N_i^{(e)} = k) \qquad \text{and} \qquad \widehat{g_k}^{(n)} = \frac{1}{n} \sum_{i=1}^n 1(D_i^{(e)} = k) \qquad k = 0, 1, 2, \dots.$$
The following result is the analogue of Proposition \ref{P.RepeatedDistr} for the erased model.

\begin{prop} \label{P.ErasedDistr}
For the erased directed configuration model with bi-degree-sequence $({\bf N}, {\bf D})$, as constructed in Subsection \ref{SS.Algorithm}, we have: 
\begin{enumerate}
\item $h^{(n)}(i,j) \to f_i g_j$  as $n \to \infty$, $i,j = 0, 1, 2, \dots$, and

\item for all $k = 0, 1, 2, \dots$, 
$$\widehat{f_k}^{(n)} \stackrel{P}{\longrightarrow} f_k \qquad \text{and} \qquad\widehat{g_k}^{(n)} \stackrel{P}{\longrightarrow} g_k, \qquad n \to \infty.$$
\end{enumerate}
\end{prop}

\section{Proofs}

In this section we give the proofs of all the results in the paper. We divide the proofs into two subsections, one containing those belonging to Section \ref{S.DegreeSequences} and those belonging to Section \ref{S.ConfigurationModel}. Throughout the remainder of the paper we use the following notation: $g(x) \sim f(x)$ if $\lim_{x \to \infty} g(x)/f(x) = 1$, $g(x) = O (f(x))$ if $\limsup_{x \to \infty} g(x)/f(x) < \infty$, and $g(x) = o( f(x))$ if $\lim_{x \to \infty} g(x)/f(x) = 0$.

\subsection{Degree Sequences} \label{SS.DegreeSequencesProofs}

This subsection contains the proofs of Lemma \ref{L.DifferenceLimit}, Theorems \ref{thm:always-graphical} and \ref{T.WeakConvergence}, and Proposition \ref{P.MainConditions}. 

\begin{proof}[Proof of Lemma \ref{L.DifferenceLimit}]
Let $Z_i = \gamma_i - \xi_i$ and note that the $\{Z_i\}$ are i.i.d. mean zero random variables. If $E[ Z_1^2 ] < \infty$, then Chebyshev's inequality gives
$$P(\mathcal{D}_n^c) = P\left( \left| \sum_{i=1}^n Z_i \right| > n^{1/2+\delta_0} \right) \leq \frac{n \var(Z_1)}{n^{1+2\delta_0}} = O( n^{-2\delta_0}) = o(1)$$
as $n \to \infty$. 

Suppose now that $E[Z_1^2] = \infty$, which implies that $\kappa = 1 - \max\{\alpha^{-1}, \beta^{-1}\} \in (0, 1/2]$. Let $\theta = \max\{\alpha^{-1}, \beta^{-1}\}$, define $t_n = n^{\theta + \epsilon}$, $0 < \epsilon < \min\{\delta_0, \theta^{-1} - \theta\}$, and let $\{\tilde Z_i\}$ be a sequence of i.i.d. random variables having distribution $P(\tilde Z_1 \leq x) = P(Z_1 \leq x | |Z_1| \leq t_n)$.    Then, 
\begin{align*}
&P\left(  \left| \sum_{i=1}^n Z_i \right| > n^{1-\kappa+\delta_0} \right) \\
&= P\left( \left| \sum_{i=1}^n \tilde Z_i \right|  > n^{1-\kappa+\delta_0} \right) P(|Z_1| \leq t_n)^n  + P\left( \left| \sum_{i=1}^n Z_i \right| > n^{1-\kappa+\delta_0}, \, \max_{1\leq i \leq n} |Z_i| > t_n  \right) \\
&\leq P\left( \left| \sum_{i=1}^n \tilde Z_i - nE[\tilde Z_1] \right| + n|E[\tilde Z_1]|  > n^{1-\kappa+\delta_0} \right) + P\left( \max_{1\leq i \leq n} |Z_i| > t_n  \right) .
\end{align*}
By the union bound,
\begin{align*}
P\left( \max_{1\leq i \leq n} |Z_i| > t_n  \right) &\leq n P(|Z_1| > t_n) \leq nP(\gamma_1 + \xi_1 > t_n) \leq nP(\gamma_1 > t_n/2) + nP(\xi_1 > t_n/2) \\
&\leq n (t_n/2)^{-\alpha} L_F(t_n/2) + n (t_n/2)^{-\beta} L_G(t_n/2) \\
&= O \left( n^{1-\alpha(\theta+\epsilon)} L_F(t_n) + n^{1-\beta(\theta+\epsilon)} L_G(t_n) \right) \\
&= O\left( n^{-\alpha\epsilon} L_F(t_n) + n^{-\beta\epsilon} L_G(t_n) \right) 
\end{align*}
as $n \to \infty$, which converges to zero by basic properties of slowly varying functions (see, e.g., Proposition~1.3.6 in \cite{BiGoTe_1987}). Next, note that since $E[Z_1] = 0$, 
$$|E[ \tilde Z_1 ]| = \frac{| E[ Z_1 1(|Z_1| > t_n)] |}{P(|Z_1| \leq t_n)} \leq \frac{E[ |Z_1| 1(|Z_1| > t_n)] }{P(|Z_1| \leq t_n)} \leq (1+o(1))  \left( t_n P(|Z_1| > t_n) + \int_{t_n}^\infty P(|Z_1| > z) dz \right).$$
To estimate the integral note that
\begin{align*}
\int_{t_n}^\infty P(|Z_1| > z) dz &\leq \int_{t_n}^\infty ( P(\gamma_1 > z/2) + P(\xi_1 > z/2) ) dz \\
&\leq 2\int_{t_n/2}^\infty \left( u^{-\alpha} L_F(u) + u^{-\beta} L_G(u) \right) du \\
&\sim 2 \left(  (\alpha-1)^{-1} (t_n/2)^{-\alpha+1} L_F(t_n/2) + (\beta-1)^{-1} (t_n/2)^{-\beta+1} L_G(t_n/2) \right) \\
&= O\left( n^{-(\alpha-1)(\theta+\epsilon)} L_F(t_n) + n^{-(\beta-1)(\theta+\epsilon)} L_G(t_n) \right) ,
\end{align*}
where in the third step we used Proposition 1.5.10 in \cite{BiGoTe_1987}. Now note that
$$\min\{ (\alpha-1)(\theta+\epsilon) , (\beta-1)(\theta+\epsilon) \} = (\theta^{-1} -1)(\theta+\epsilon) = \kappa +\epsilon (\theta^{-1} - 1),$$
from where it follows that
$$|E[ \tilde Z_1 ]| = O\left( n^{-\kappa - \epsilon(\theta^{-1}-1)} (L_F(t_n) + L_G(t_n))  \right) = o\left( n^{-\kappa + \delta_0}  \right)$$ 
as $n \to \infty$. In view of this, we can use Chebyshev's inequality to obtain
\begin{equation}\label{eq:TruncChebyshev}
P\left( \left| \sum_{i=1}^n \tilde Z_i - nE[\tilde Z_1] \right| + n|E[\tilde Z_1]|  > n^{1-\kappa+\delta_0} \right) \leq \frac{\var(\tilde Z_1)}{n^{1-2(\kappa-\delta_0)} (1+o(1))}.
\end{equation}
Finally, to see that this last bound converges to zero note that
\begin{align*}
\var(\tilde Z_1) &\leq E[ \tilde Z_1^2 ] = \frac{1}{P(|Z_1| \leq t_n)} E[ Z_1^2 1(|Z_1| \leq t_n)] \leq (1+o(1)) E\left[ |Z_1|^{\theta^{-1}-\epsilon}  \right] t_n^{2-\theta^{-1}+\epsilon},
\end{align*}
where $E[ |Z_1|^{\theta^{-1}-\epsilon} ] < \infty$ by the remark following \eqref{eq:powerlaw-majorant}. We conclude that \eqref{eq:TruncChebyshev} is of order
$$O\left(  t_n^{2-\theta^{-1}+\epsilon} n^{2(\kappa-\delta_0) - 1} \right) = O\left( n^{(\theta+\epsilon) (2-\theta^{-1}+\epsilon) + 2(\kappa-\delta_0) - 1} \right) = o\left( n^{-2(\delta_0 - \epsilon)} \right) = o(1)$$
as $n \to \infty$. This completes the proof. 
\end{proof}

Before giving the proof of Theorem \ref{thm:always-graphical} we will need the following preliminary lemma.

\begin{lemma} \label{L.SumsOrderStats}
Let $\{X_1, \dots, X_n\}$ be an i.i.d. sequence of nonnegative random variables having distribution function $V$, and let $X^{(i)}$ denote the $i$th order statistic. Then, for any $k \leq n$,
$$\sum_{i=n-k+1}^n E\left[ X^{(i)} \right] \leq \int_0^\infty \min\left\{ n \overline{V}(x), k \right\} dx.$$
\end{lemma}

\begin{proof}
Note that
$$E\left[ X^{(i)} \right] = \int_0^\infty P(X^{(i)} > x) \, dx = \int_0^\infty \sum_{j = n-i+1}^n \binom{n}{j} \overline{V}(x)^j V(x)^{n-j} dx,$$
from where it follows 
\begin{align*}
\sum_{i=n-k+1}^n E\left[ X^{(i)} \right] &= \sum_{i=n-k+1}^n \sum_{j = n-i+1}^n \binom{n}{j} \int_0^\infty  \overline{V}(x)^j V(x)^{n-j} dx \\
&=  \sum_{j=1}^{n} \min\{j, k\} \binom{n}{j} \int_0^\infty \overline{V}(x)^{j} V(x)^{n-j} dx \\
&= \int_0^\infty E\left[  \min\{ B(n, \overline{V}(x)), k \} \right] dx,
\end{align*}
where $B(n,p)$ is a Binomial$(n,p)$ random variable. Since the function $u(t) = \min\{t, k\}$ is concave, Jensen's inequality gives
$$E\left[  \min\{ B(n, \overline{V}(x)), k \} \right] \leq \min\left\{ E[ B(n, \overline{V}(x)) ], k \right\} = \min\left\{ n\overline{V}(x), k \right\}.$$
\end{proof}

\begin{proof}[Proof of Theorem \ref{thm:always-graphical}]
Since by construction $\sum_{i=1}^n N_i = \sum_{i=1}^n D_i$, it follows from Theorem~\ref{thm:graphical-condition} that it suffices to show that
\[
\lim_{n \to \infty} P\left( \max_{A\subseteq V} \left(\sum_{v_{i} \in A} N_{i} - \sum_{i=1}^{n} \min\{D_{i}, |A-\{v_{i}\}|\} \right) > 0 \right) = 0. 
\]
Fix $0 < \epsilon < \min\{ \beta -1, \alpha-1, 1/2\}$ and use the union bound to obtain
\begin{align}
&P\left( \max_{A\subseteq V} \left(\sum_{v_{i} \in A} N_{i} - \sum_{i=1}^{n} \min\{D_{i}, |A-\{v_{i}\}|\} \right) > 0 \right)  \notag \\
&\leq P\left( \max_{A\subseteq V, |A| \leq n^{(1+\epsilon)/\beta} } \left(\sum_{v_{i} \in A} N_{i} - \sum_{i=1}^{n} \min\{D_{i}, |A-\{v_{i}\}|\} \right) > 0 \right) \label{eq:smallSets} \\
&\hspace{5mm} + P\left( \max_{A\subseteq V, |A| > n^{(1+\epsilon)/\beta} } \left(\sum_{v_{i} \in A} N_{i} - \sum_{i=1}^{n} \min\{D_{i}, |A-\{v_{i}\}|\} \right) > 0 \right). \label{eq:bigSets}
\end{align}
By conditioning on how many of the $D_i$ are larger than $n^{(1+\epsilon)/\beta}$ we obtain that \eqref{eq:bigSets} is bounded by
\begin{align*}
&P\left( \max_{A\subseteq V, |A| > n^{(1+\epsilon)/\beta} } \left(\sum_{v_{i} \in A} N_{i} - \sum_{i=1}^{n} \min\{D_{i}, |A-\{v_{i}\}|\} \right) > 0, \, \max_{1 \leq i \leq n} D_i \leq n^{(1+\epsilon)/\beta} \right) \\
&\hspace{5mm} + P\left(  \max_{1 \leq i \leq n} D_i > n^{(1+\epsilon)/\beta} \right) \\
&\leq P\left( \max_{A\subseteq V, |A| > n^{(1+\epsilon)/\beta} } \left(\sum_{v_{i} \in A} N_{i} - \sum_{i=1}^{n} D_{i}  \right) > 0 \right) + P\left(  \max_{1 \leq i \leq n} D_i > n^{(1+\epsilon)/\beta} \right) \\
&= P\left( \left.  \max_{1 \leq i \leq n} (\xi_i + \chi_i)  > n^{(1+\epsilon)/\beta} \right| \mathcal{D}_n \right),
\end{align*}
where $\mathcal{D}_n$ was defined in Lemma \ref{L.DifferenceLimit}. Now note that by the union bound we have
\begin{align*}
P\left( \left.  \max_{1 \leq i \leq n} (\xi_i + \chi_i)  > n^{(1+\epsilon)/\beta} \right| \mathcal{D}_n \right) &\leq \frac{1}{P(\mathcal{D}_n)} \cdot P\left( \max_{1 \leq i \leq n} (\xi_i + \chi_i)  > n^{(1+\epsilon)/\beta} \right) \\
&\leq \frac{1}{P(\mathcal{D}_n)} \sum_{i=1}^n P\left( \xi_i + \chi_i > n^{(1+\epsilon)/\beta} \right)  \\
&\leq \frac{1}{P(\mathcal{D}_n)} \cdot n \left( n^{(1+\epsilon)/\beta} -1 \right)^{-\beta} L_G \left( n^{(1+\epsilon)/\beta} -1 \right) \\
&= O\left( n^{-\epsilon} L_G \left( n^{(1+\epsilon)/\beta} \right) \right) = o(1) ,
\end{align*}
as $n \to \infty$, where the last step follows from Lemma \ref{L.DifferenceLimit} and basic properties of slowly varying functions (see, e.g., Chapter 1 in \cite{BiGoTe_1987}).  

Next, to analyze \eqref{eq:smallSets} let $k_n = \lfloor n^{(1+\epsilon)/\beta} \rfloor$ and note that we can write it as
\begin{align*}
&P\left( \max_{A\subseteq V,   |A| \leq k_n } \left(\sum_{v_{i} \in A} N_{i} - \sum_{i=1}^{n} \min\{D_{i}, |A-\{v_{i}\}|\} \right) > 0 \right) \\
&\leq P\left( \max\left\{ \max_{A\subseteq V,  \, 2 \leq |A| \leq k_n } \left(\sum_{v_{i} \in A} N_{i} - \sum_{i=1}^{n} \min\{D_{i}, 1 \} \right), \,   \max_{1 \leq j \leq n} \left( N_{j} - \sum_{i=1}^{n} \min\{D_{i}, |\{ v_j \} -\{v_{i}\}|\} \right) \right\} > 0 \right)  \\
&= P\left( \max\left\{ \sum_{i=n - k_n + 1}^n N^{(i)}  , \,  (N+D)^{(n)}  \right\} - \sum_{i=1}^{n} \min\{D_{i}, 1 \}  > 0 \right),
\end{align*}
where $x^{(i)}$ is the $i$th smallest of $\{x_1, \dots, x_n\}$. Now let $a_0 = E[ \min\{ \xi_1, 1 \} ] = \overline{G}(0) > 0$ and split the last probability as follows
\begin{align}
&P\left( \max\left\{ \sum_{i=n - k_n + 1}^n  N^{(i)}  , \,  (N+D)^{(n)}  \right\} - \sum_{i=1}^{n} \min\{D_{i}, 1 \}  > 0 \right) \notag \\
&\leq P\left( \max\left\{ \sum_{i=n - k_n+ 1}^n  N^{(i)}  , \,  (N+D)^{(n)}  \right\}   >  a_0 n - n^{1/2+\epsilon}, \,  \sum_{i=1}^{n} \min\{D_{i}, 1 \} \geq a_0 n - n^{1/2+\epsilon} \right) \label{eq:twoMaximums} \\
&\hspace{5mm} + P\left( \sum_{i=1}^{n} \min\{D_{i}, 1 \} < a_0 n - n^{1/2+\epsilon} \right). \label{eq:Chebyshev}
\end{align}
To bound \eqref{eq:Chebyshev} use $D_i \geq \xi_i$ for all $i=1,\dots, n$ and Chebyshev's inequality to obtain
\begin{align*}
P\left( \sum_{i=1}^{n} \min\{D_{i}, 1 \} < a_0 n - n^{1/2+\epsilon} \right) &\leq \frac{1}{P(\mathcal{D}_n)} P\left( \sum_{i=1}^n (a_0 - \min\{ \xi_i, 1\}) > n^{1/2+\epsilon} \right) \\
&\leq \frac{n \var(\min\{ \xi_1, 1\})}{P(\mathcal{D}_n) n^{1+2\epsilon}} = O\left( n^{-2\epsilon} \right),
\end{align*}
while the union bound gives that \eqref{eq:twoMaximums} is bounded by
\begin{equation*}
 P\left( \max\left\{ \sum_{i=n - k_n+ 1}^n  N^{(i)}  , \,  (N+D)^{(n)}  \right\}   >  b_n \right) \leq P\left(  \sum_{i=n - k_n + 1}^n  N^{(i)} > b_n \right) + P\left( (N+D)^{(n)} > b_n \right),
\end{equation*}
where $b_n = a_0 n - n^{1/2+\epsilon}$. For the second probability the union bound again gives
\begin{align*}
P\left( (N+D)^{(n)} > b_n \right) &\leq P\left( N^{(n)} > b_n/2 \right) + P\left( D^{(n)} > b_n/2 \right) \\
&\leq \frac{n}{P(\mathcal{D}_n)} \left( P(\gamma_1 +\tau_1> b_n/2) + P( \xi_1 + \chi_1> b_n/2) \right) \\
&\leq \frac{n}{P(\mathcal{D}_n)} \left(  (b_n/2-1)^{-\alpha} L_F(b_n/2-1) + (b_n/2-1)^{-\beta} L_G(b_n/2-1) \right) \\
&= O\left( n^{-\alpha+1} L_F(n) + n^{-\beta+1} L_G(n)  \right) = o(1)
\end{align*}
as $n \to \infty$. Finally, by Markov's inequality and Lemma \ref{L.SumsOrderStats},
\begin{align*} 
P\left(  \sum_{i=n - k_n + 1}^n  N^{(i)} > b_n \right) &\leq \frac{1}{b_n}  \sum_{i=n - k_n + 1}^n  E\left[ N^{(i)} \right] \leq \frac{1}{b_n P(\mathcal{D}_n)} \sum_{i=n-k_n+1}^n E[ \gamma^{(i)} + 1] \\
&\leq \frac{1}{b_n P(\mathcal{D}_n)} \left(  \int_0^\infty \min\left\{ n \overline{F}(x), k_n \right\} dx + k_n \right) \\
&= a_0^{-1} (1 + o(1)) \int_0^\infty \min\left\{ \overline{F}(x), n^{(1+\epsilon)/\beta - 1} \right\} dx + o(1)  \\
&\leq a_0^{-1} (1+o(1)) \left( n^{(1+\epsilon)/\beta-1} + \int_1^\infty \min\left\{ K x^{-\alpha+\epsilon} , n^{(1+\epsilon)/\beta-1} \right\} dx \right) + o(1) \\
&= o(1) + O\left( \int_1^\infty \min\left\{ x^{-\alpha+\epsilon}, n^{(1+\epsilon)/\beta-1} \right\} dx \right)
\end{align*}
as $n \to \infty$, where $K = \sup_{t \geq 1} t^{-\epsilon} L_F(t) < \infty$. Since
\begin{align*}
\int_1^\infty \min\left\{ x^{-\alpha+\epsilon}, n^{(1+\epsilon)/\beta-1} \right\} dx &= n^{(1+\epsilon)/\beta-1} (n^{(\beta-1-\epsilon)/(\beta(\alpha-\epsilon))}-1) + \int_{n^{(\beta-1-\epsilon)/(\beta(\alpha-\epsilon))} }^\infty x^{-\alpha+\epsilon} dx \\
&= O\left( n^{-(\beta-1-\epsilon)(\alpha-1-\epsilon)/(\beta(\alpha-\epsilon))} \right) = o(1),
\end{align*}
the proof is complete.
\end{proof}

The last two proofs of this section are those of Theorem \ref{T.WeakConvergence} and Proposition \ref{P.MainConditions}.

\begin{proof}[Proof of Theorem \ref{T.WeakConvergence}]
Let $u: \mathbb{N}^{r+s} \to [-M, M]$, $M > 0$, be a continuous bounded function, and let $\Delta_n, \mathcal{D}_n$ be defined as in Lemma \ref{L.DifferenceLimit}. Then,
\begin{align}
&\left| E\left[ u(N_{i_1}, \dots, N_{i_r}, D_{j_1}, \dots, D_{j_s}) \right] - E\left[ u(\eta_1, \dots, \eta_r, \xi_1, \dots, \xi_s) \right] \right| \notag \\
&= \left|  E\left[ u(\gamma_{i_1}+\tau_{i_1}, \dots, \gamma_{i_r}+\tau_{i_r}, \xi_{j_1} + \chi_{j_1}, \dots, \xi_{j_s} + \chi_{j_s}) | \mathcal{D}_n \right] - E\left[ u(\gamma_{i_1}, \dots, \gamma_{i_r}, \xi_{j_1}, \dots, \xi_{j_s}) \right]   \right| \notag \\
&\leq  \left|   E\left[ u(\gamma_{i_1}+\tau_{i_1}, \dots, \gamma_{i_r}+\tau_{i_r}, \xi_{j_1} + \chi_{j_1}, \dots, \xi_{j_s} + \chi_{j_s}) -u(\gamma_{i_1}, \dots, \gamma_{i_r}, \xi_{j_1}, \dots, \xi_{j_s}) | \mathcal{D}_n \right]   \right| \label{eq:noChi} \\
&\hspace{5mm} + \left|   E\left[ u(\gamma_{i_1}, \dots, \gamma_{i_r}, \xi_{j_1}, \dots, \xi_{j_s}) | \mathcal{D}_n \right] - E\left[ u(\gamma_{1}, \dots, \gamma_{r}, \xi_{1}, \dots, \xi_{s}) \right]   \right|. \label{eq:Conditional}
\end{align}
Let $T = \sum_{t=1}^r \tau_{i_t} + \sum_{t=1}^s \chi_{j_s}$. Since $u$ is bounded then \eqref{eq:noChi} is smaller than or equal to
\begin{align*}
&E\left[ \left. \left| u(\gamma_{i_1}+\tau_{i_1}, \dots, \gamma_{i_r}+\tau_{i_r}, \xi_{j_1} + \chi_{j_1}, \dots, \xi_{j_s} + \chi_{j_s}) -u(\gamma_{i_1}, \dots, \gamma_{i_r}, \xi_{j_1}, \dots, \xi_{j_s}) \right| 1\left( T \geq 1 \right) \right| \mathcal{D}_n \right] \\
&\leq 2M P\left( \left. T \geq 1 \right| \mathcal{D}_n \right) \leq 2M \left( \sum_{t = 1}^r P( \tau_{i_t} = 1 | \mathcal{D}_n)  + \sum_{t=1}^s P(\chi_{j_t} = 1| \mathcal{D}_n) \right) \\
&= \frac{2M}{P(\mathcal{D}_n)} \left( \sum_{t=1}^r E[ 1(\tau_{i_t} = 1, \mathcal{D}_n)] + \sum_{t=1}^s  E[ 1( \chi_{j_t} = 1, \mathcal{D}_n) ] \right).
\end{align*}
To compute the last expectations let $\mathcal{F}_n = \sigma(\gamma_1, \dots, \gamma_n, \xi_1, \dots, \xi_n)$ and note that
$$E[ 1(\chi_{j_t} = 1, \mathcal{D}_n) ] = E[ 1(\mathcal{D}_n) E[1(\chi_{j_t} = 1) | \mathcal{F}_n] ] = E\left[ 1(\mathcal{D}_n, \Delta_n \geq 0) \frac{\binom{n-1}{\Delta_n-1}}{\binom{n}{\Delta_n}}  \right] = E\left[ 1(\mathcal{D}_n, \Delta_n \geq 0) \frac{\Delta_n}{n}  \right],$$
and symmetrically, 
$$E[ 1(\tau_{i_t} = 1, \mathcal{D}_n) ] = E\left[ 1(\mathcal{D}_n, \Delta_n < 0) \frac{|\Delta_n|}{n}  \right],$$
from where it follows that \eqref{eq:noChi} is bounded by
$$2M \left( \sum_{t=1}^r E\left[ \left. \frac{\Delta_n}{n} 1(\Delta_n \geq 0) \right| \mathcal{D}_n \right] + \sum_{t=1}^s E\left[ \left. \frac{|\Delta_n|}{n} 1(\Delta_n < 0) \right| \mathcal{D}_n \right]  \right) \leq 2M (r+s) n^{-\kappa+\delta_0} = o(1)$$
as $n \to \infty$. To analyze \eqref{eq:Conditional} we first note that by Lemma \ref{L.DifferenceLimit}, $P(\mathcal{D}_n) \to 1$ as $n \to \infty$, hence
\begin{align*}
E\left[ u(\gamma_{i_1}, \dots, \gamma_{i_r}, \xi_{j_1}, \dots, \xi_{j_s}) | \mathcal{D}_n \right] &= \frac{1}{P(\mathcal{D}_n)} E\left[ u(\gamma_{1}, \dots, \gamma_r, \xi_{1}, \dots, \xi_s) 1(\mathcal{D}_n) \right] \\
&= E\left[ u(\gamma_{1}, \dots, \gamma_r, \xi_{1}, \dots, \xi_s) 1(\mathcal{D}_n) \right] + o(1).
\end{align*}
Therefore, \eqref{eq:Conditional} is equal to
\begin{align*}
\left| E\left[ u(\gamma_{1}, \dots, \gamma_r, \xi_{1}, \dots, \xi_s) 1(\mathcal{D}_n^c) \right] + o(1) \right| \leq M P(\mathcal{D}_n^c) + o(1) \to 0
\end{align*}
as $n \to \infty$, which completes the proof. 
\end{proof}

\begin{proof}[Proof of Proposition \ref{P.MainConditions}]
Fix $\epsilon > 0$ and let $\mathcal{D}_n = \{ |\Delta_n| \leq n^{1-\kappa+\delta_0}\}$. For the first limit fix $i, j = 0, 1, 2, \dots$ and note that by the union bound,
\begin{align*}
&P\left( \left| \frac{1}{n} \sum_{k=1}^n 1(N_k = i, D_k = j) - f_i g_j \right| > \epsilon \right) \\
&\leq P\left( \left. \left| \frac{1}{n} \sum_{k=1}^n (1(\gamma_k + \tau_k = i, \xi_k + \chi_k = j) - 1(\gamma_k = i, \xi_k = j)) \right| > \epsilon/2 \right| \mathcal{D}_n \right) \\
&\hspace{5mm} + P \left( \left. \left| \frac{1}{n} \sum_{k=1}^n 1(\gamma_k = i, \xi_k = j) - f_i g_j \right| > \epsilon/2 \right| \mathcal{D}_n \right) \\
&\leq P\left( \left. \frac{1}{n} \sum_{k=1}^n \left| 1(\gamma_k + \tau_k = i, \xi_k + \chi_k = j) - 1(\gamma_k = i, \xi_k = j)) \right| > \epsilon/2 \right| \mathcal{D}_n \right) \\
&\hspace{5mm}  + \frac{1}{P(\mathcal{D}_n) n (\epsilon/2)^2} \var( 1(\gamma_1 = i, \xi_1 = j) ),
\end{align*}
where in the last step we used Chebyshev's inequality. Clearly, $\var(1(\gamma_1 = i, \xi_1 =j) ) = f_i g_j (1-f_i g_j)$, and since by Lemma \ref{L.DifferenceLimit} $P(\mathcal{D}_n) \to 1$ as $n \to \infty$, then the second term converges to zero. To analyze the first term note that at most one of $\chi_k$ or $\tau_k$ can be one, hence,
\begin{align*}
& P\left( \left. \frac{1}{n} \sum_{k=1}^n \left| 1(\gamma_k + \tau_k = i, \xi_k + \chi_k = j) - 1(\gamma_k = i, \xi_k = j)) \right| > \epsilon/2 \right| \mathcal{D}_n \right)  \\
&\leq  P\left( \left. \frac{1}{n} \sum_{k=1}^n \left( \left|  1(\xi_k + \chi_k = j) - 1(\xi_k = j)  \right| +  \left| 1(\gamma_k + \tau_k = i) - 1(\gamma_k = i)  \right| \right) > \epsilon/2 \right| \mathcal{D}_n \right) \\
&\leq P\left( \left. \frac{1}{n} \sum_{k=1}^n \left( 1(\chi_k = 1) + 1(\tau_k = 1)  \right) > \epsilon/2 \right| \mathcal{D}_n \right)  \\
&= P\left( \left. \frac{|\Delta_n|}{n} > \epsilon/2 \right| \mathcal{D}_n \right)  \\
&\leq 1(n^{-\kappa+\delta_0} > \epsilon/2) \to 0 
\end{align*}
as $n \to \infty$. 

Next, for the average degrees we have
\begin{align} 
P\left( \left| \frac{1}{n} \sum_{i=1}^n N_i - E[\gamma_1] \right| > \epsilon \right) &= P\left( \left.  \left| \frac{1}{n} \sum_{i=1}^n (\gamma_i+ \tau_i) - E[\gamma_1] \right| > \epsilon \right| \mathcal{D}_n \right) \notag \\
&\leq P\left( \left. \left| \frac{1}{n} \sum_{i=1}^n \gamma_i - E[\gamma_1] \right| + \frac{|\Delta_n|}{n}  > \epsilon \right| \mathcal{D}_n \right) \notag \\
&\leq \frac{1}{P(\mathcal{D}_n)} P\left( \left| \frac{1}{n} \sum_{i=1}^n \gamma_i - E[\gamma_1] \right| + n^{-\kappa+\delta_0}  > \epsilon  \right)  ,  \label{eq:in-degreeWL}
\end{align}
symmetrically,
\begin{align}
P\left( \left| \frac{1}{n} \sum_{i=1}^n D_i - E[\xi_1] \right| > \epsilon \right) &\leq \frac{1}{P(\mathcal{D}_n)} P\left( \left| \frac{1}{n} \sum_{i=1}^n \xi_i  - E[\xi_1] \right| + n^{-\kappa+\delta_0} > \epsilon   \right) , \label{eq:out-degreeWL}
\end{align}
and since $\tau_i \chi_i = 0$ for all $1 \leq i \leq n$, 
\begin{align}
P\left( \left| \frac{1}{n} \sum_{i=1}^n N_i D_i - E[\gamma_1 \xi_1] \right| > \epsilon \right) &= P\left( \left. \left| \frac{1}{n} \sum_{i=1}^n (\gamma_i \xi_i + \tau_i \xi_i + \gamma_i \chi_i - E[\gamma_1 \xi_1] \right|   > \epsilon \right| \mathcal{D}_n \right) \notag \\
&\leq P\left( \left. \left| \frac{1}{n} \sum_{i=1}^n \gamma_i \xi_i  - E[\gamma_1 \xi_1] \right| + \sum_{i=1}^n ( \tau_i \xi_i + \gamma_i \chi_i)    > \epsilon \right| \mathcal{D}_n \right) \notag \\
&\leq \frac{1}{P(\mathcal{D}_n)} P\left(  \left| \frac{1}{n} \sum_{i=1}^n \eta_i \xi_i  - E[\gamma_1 \xi_1] \right|  + n^{-\kappa+\delta} > \epsilon   \right) \label{eq:mix-degreesWL} \\
&\hspace{5mm} +   P\left( \left. \frac{1}{n} \sum_{i=1}^n ( \tau_i \xi_i + \gamma_i \chi_i) > n^{-\kappa+\delta} \right| \mathcal{D}_n  \right), \label{eq:chiANDeta}
\end{align}
for any $\delta_0 < \delta < \kappa$. By Lemma \ref{L.DifferenceLimit}, $P(\mathcal{D}_n)$ converges to one, and by the Weak Law of Large Numbers (WLLN) we have that each of \eqref{eq:in-degreeWL}, \eqref{eq:out-degreeWL} and \eqref{eq:mix-degreesWL} converges to zero as $n \to \infty$, as required. To see that \eqref{eq:chiANDeta} converges to zero use Markov's inequality to obtain 
\begin{align}
P\left( \left. \frac{1}{n} \sum_{i=1}^n ( \tau_i \xi_i + \gamma_i \chi_i)  > n^{-\kappa+\delta} \right| \mathcal{D}_n  \right) \leq \frac{E[ \tau_1 \xi_1 + \gamma_1 \chi_1 | \mathcal{D}_n]}{n^{-\kappa+\delta}} = \frac{E[ (\tau_1 \xi_1 + \gamma_1 \chi_1) 1(\mathcal{D}_n)]}{P(\mathcal{D}_n) n^{-\kappa+\delta}}. \label{eq:mixedProducts}
\end{align}
Now let $\mathcal{F}_n = \sigma(\gamma_1, \dots, \gamma_n, \xi_1, \dots, \xi_n)$ to compute
$$E[ (\tau_1 \xi_1 + \gamma_1 \chi_1) 1(\mathcal{D}_n)] = E[ (\xi_1 E[\tau_1 | \mathcal{F}_n] + \gamma_1E[\chi_1| \mathcal{F}_n]) 1(\mathcal{D}_n) ] \leq E\left[ \left( \xi_1 + \gamma_1  \right) \frac{|\Delta_n|}{n} 1(\mathcal{D}_n) \right] \leq 2\mu n^{-\kappa+\delta_0},$$
which implies that \eqref{eq:mixedProducts} converges to zero. 

Finally, provided that $E[\gamma_1^2 + \xi_1^2] < \infty$, the WLLN combined with the arguments used to bound \eqref{eq:chiANDeta} give
\begin{align*}
P\left( \left| \frac{1}{n} \sum_{i=1}^n N_i^2 - E[\gamma_1^2] \right| > \epsilon \right) &\leq \frac{1}{P(\mathcal{D}_n)} P\left( \left| \frac{1}{n} \sum_{i=1}^n \gamma_i^2 - E[\gamma_1^2] \right| +  \frac{1}{n} \sum_{i=1}^n (2\gamma_i \tau_i + \tau_i^2) > \epsilon, \mathcal{D}_n \right)  \\
&\leq  \frac{1}{P(\mathcal{D}_n)} P\left( \left| \frac{1}{n} \sum_{i=1}^n \gamma_i^2 - E[\gamma_1^2] \right| + n^{-\kappa+\delta} > \epsilon \right) \\
&\hspace{5mm} +  P\left( \left.  \frac{1}{n} \sum_{i=1}^n (2\gamma_i \tau_i + \tau_i^2) > n^{-\kappa+\delta} \right| \mathcal{D}_n \right) \\
&\leq o(1) + \frac{E[ (2\gamma_1 +1)\tau_1| \mathcal{D}_n]}{n^{-\kappa+\delta}} \\
&\leq o(1) + \frac{E[2\gamma_1+1]}{P(\mathcal{D}_n) n^{\delta-\delta_0}},
\end{align*}
and symmetrically, 
$$P\left( \left| \frac{1}{n} \sum_{i=1}^n D_i^2 - E[\xi_1^2] \right| > \epsilon \right) \to 0,$$
as $n \to \infty$. 
\end{proof}

\subsection{Configuration Model} \label{SS.ConfigurationModelProofs}

This subsection contains the proofs of Proposition \ref{P.Uniformity}, which establishes the uniformity of simple graphs, Propositions \ref{P.PoissonLimit} and \ref{P.RepeatedDistr}, which concern the repeated directed configuration model, and Proposition \ref{P.ErasedDistr} which refers to the erased directed configuration model. 

\begin{proof}[Proof of Proposition \ref{P.Uniformity}]
Suppose ${\bf m}$ and ${\bf d}$ have equal sum $l_n$, and number the inbound and outbound half-edges by $1, 2 , \dots, l_n$. The process of matching half edges in the configuration model is equivalent to a permutation $(p(1), p(2), \dots, p(l_n))$ of the numbers $(1, 2, \dots, l_n)$ where we pair the $i$th inbound half-edge to the $p(i)$th outbound half-edge, with all $l_n!$ permutations being equally likely. Note that different permutations can actually lead to the same graph, for example, if we switch the position of two outbound half-edges of the same node, so not all multigraphs have the same probability. Nevertheless, a simple graph can only be produced by $\prod_{i=1}^n d_i! m_i!$ different permutations; to see this note that for each node $v_i$, $i = 1, \dots, n$, we can permute its $m_i$ inbound half-edges and its $d_i$ outbound half-edges without changing the graph. It follows that since the number of permutations leading to a simple graph is the same for all simple graphs, then conditional on the resulting graph being simple, it is uniformly chosen among all simple graphs having bi-degree-sequence $({\bf m}, {\bf d})$.
\end{proof}

Next, we give the proofs of the results related to the repeated directed configuration model. Before proceeding with the proof of Proposition \ref{P.PoissonLimit} we give the following preliminary lemma, which will be used to establish that under Condition \ref{Cond.RegularityDirected} the maximum in- and out-degrees cannot grow too fast.

\begin{lemma} \label{L.MaximumGrowth}
Let $\left\{ a_{nk}: 1 \leq k \leq n, n \in \mathbb{N} \right\}$ be a triangular array of nonnegative integers, and suppose there exist nonnegative numbers $\{ p_j: j \in \mathbb{N} \cup \{0\}\}$ such that $\sum_{j=0}^\infty p_j = 1$,
\[
\lim_{n \to \infty} \frac{1}{n} \sum_{k=1}^n 1(a_{nk} = j) = p_j, \quad \text{for all }  j \in \mathbb{N} \cup \{0\}  \qquad \text{and} \qquad \lim_{n\rightarrow\infty} \frac{1}{n}\sum_{k=1}^n a_{nk} = \sum_{j=0}^\infty j p_j < \infty.
\]
Then, 
\[ \lim_{n\rightarrow \infty}\max_{1\leq k\leq n}\frac{a_{nk}}{n}= 0. \]
\end{lemma}

\begin{proof} 
Define 
\begin{align*}
F(x) = \sum_{j = 0}^{\lfloor x \rfloor} p_j \qquad \text{and} \qquad F_n(x) = \frac{1}{n} \sum_{k=1}^n 1(a_{nk} \leq x)
\end{align*}
and note that $F$ and $F_n$ are both distribution functions with support on the nonnegative integers. Define the pseudoinverse operator $h^{-1}(u) = \inf\{ x \geq 0: u \leq h(x) \}$ and let
$$X_n = F_n^{-1}(U) \qquad \text{and} \qquad X = F^{-1}(U),$$
where $U$ is a Uniform(0,1) random variable. It is easy to verify that $X_n$ and $X$ have distributions $F_n$ and $F$, respectively. Furthermore, the assumptions imply that
$$X_n \to X \qquad \text{a.s.}$$
as $n \to \infty$ and
$$E[X_n] = \sum_{j=0}^\infty j \frac{1}{n} \sum_{k=1}^n 1(a_{nk} = j) = \frac{1}{n} \sum_{k=1}^n \sum_{j=0}^\infty j 1(a_{nk} = j) = \frac{1}{n} \sum_{k=1}^n a_{nk} \to E[X]$$
as $n \to \infty$, where the exchange of sums is justified by Fubini's theorem. Now note that by Fatou's lemma,
$$\liminf_{n \to \infty} E[ X_n 1(X_n \leq \sqrt{n}) ] \geq E\left[ \liminf_{n \to \infty} X_n 1(X_n \leq \sqrt{n}) \right] = E[X],$$
which implies that
$$\lim_{n \to \infty} E[ X_n 1(X_n > \sqrt{n}) ] = 0.$$
Finally, 
$$E[ X_n 1(X_n \geq n)] = \sum_{j= \lfloor \sqrt{n} \rfloor +1}^\infty j \frac{1}{n} \sum_{k=1}^n 1(a_{nk} = j) = \frac{1}{n} \sum_{k=1}^n \sum_{j = \lfloor \sqrt{n} \rfloor +1}^\infty j 1(a_{nk} = j) = \frac{1}{n} \sum_{k=1}^n a_{nk} 1(a_{nk} > \sqrt{n}),$$ 
from where it follows that
$$\lim_{n \to \infty} \max_{1 \leq k \leq n} \frac{a_{nk} 1(a_{nk} > \sqrt{n})}{n} = 0,$$
which in turn implies that
$$\lim_{n \to \infty} \max_{1 \leq k \leq n} \frac{a_{nk}}{n} \leq \lim_{n\to\infty} \left( \frac{\sqrt{n}}{n} + \max_{1 \leq k \leq n} \frac{a_{nk} 1(a_{nk} > \sqrt{n})}{n} \right) = 0.$$
\end{proof}

\begin{proof}[Proof of Proposition \ref{P.PoissonLimit}]
Following the proof of Proposition 7.9 in \cite{van2009random}, we define the random variable $\tilde M_n$ to be the total number of pairs of multiple edges in the same direction, e.g., if from node $v_i$ to node $v_j$ there are $k \geq 2$ edges, their contribution to $\tilde M_n$ is $\binom{k}{2}$. Note that $M_n \leq \tilde M_n$, with strict inequality whenever there is at least one pair of nodes having three or more multiple edges in the same direction. We claim that $\tilde M_n - M_n \stackrel{P}{\longrightarrow} 0$ as $n \to \infty$, which implies that 
$$\text{if} \qquad (S_n, \tilde M_n) \Rightarrow (S, M), \qquad \text{then} \qquad (S_n, M_n) \Rightarrow (S,M)$$
as $n \to \infty$. To prove the claim start by defining indicator random variables for each of the possible self-loops and multiple edges in the same direction that the multigraph can have. For the self-loops we use the notation ${\bf u} = (r,t,i)$ to define
\begin{align*}
I_{\bf u} := 1(\text{self-loop from the $r$th outbound stub to the $t$th inbound stub of node $v_i$}),
\end{align*}
and for the pairs of multiple edges in the same direction we use ${\bf w} = (r_1, t_1, r_2, t_2, i,j)$ to define
\begin{align*}
J_{\bf w} &:= 1(\text{$r_s$th outbound stub of node $v_i$ paired to $t_s$th inbound stub of node $v_j$, $s = 1,2$}).
\end{align*}
The sets of possible vectors ${\bf u}$ and ${\bf w}$ are given by
\begin{align*}
\mathcal{I} &=\{(r,t,i):1\leq i\leq n ,1\leq r\leq d_{ni},1\leq t\leq m_{ni}\}, \quad \text{and} \\
\mathcal{J} &=\{(r_{1},t_{1},r_{2},t_{2},i,j):
1\leq i\neq j\leq n,1\leq r_{1}<r_{2}\leq d_{ni},1\leq t_{1}\neq t_{2}\leq m_{nj}\},
\end{align*}
respectively. It follows from this notation that
$$S_n = \sum_{{\bf u} \in \mathcal{I}} I_{\bf u} \qquad \text{and} \qquad \tilde M_n = \sum_{{\bf w} \in \mathcal{J}} J_{\bf w}.$$

Next, note that by the union bound, 
\begin{align*}
P\left( \tilde M_n - M_n \geq 1 \right) &\leq P\left( \text{at least two nodes with three or more edges in the same direction}  \right) \\
&\leq \sum_{1 \leq i \neq j \leq n} P\left( \text{three or more edges from node $v_i$ to node $v_j$} \right) \\
&\leq \sum_{1 \leq i \neq j \leq n} \frac{d_{ni}(d_{ni}-1)(d_{ni}-2)m_{nj}(m_{nj}-1)(m_{nj}-2)}{l_n(l_n-1)(l_n-2)} \\
&\leq \left( \frac{1}{\sqrt{n}} \max_{1\leq i \leq n} d_{ni} \right) \left( \frac{1}{\sqrt{n}} \max_{1 \leq j \leq n} m_{nj} \right) \left(\frac{n}{l_n-2}\right)^3 \cdot \frac{1}{n} \sum_{i=1}^n d_{ni}^2 \cdot \frac{1}{n} \sum_{j=1}^n m_{nj}^2 \\
&= o(1)
\end{align*}
as $n \to \infty$, where for the last step we used Condition \ref{Cond.RegularityDirected} and Lemma \ref{L.MaximumGrowth}.  It follows that $\tilde M_n - M_n \stackrel{P}{\longrightarrow} 0$ as claimed.

We now proceed to prove that $(S_n, \tilde M_n) \Rightarrow (S, M)$, where $S$ and $M$ are independent Poisson random variables with means $\lambda_1$ and $\lambda_2$, respectively.  To do this we use Theorem 2.6 in \cite{van2009random} which says that if for any $p, q \in \mathbb{N}$
$$\lim_{n \to \infty} E\left[ (S_n)_p (\tilde M_n)_q \right] = \lambda_1^p \lambda_2^q,$$
where $(X)_r = X (X-1) \cdots (X-r+1)$, then $(S_n, \tilde M_n) \Rightarrow (S,M)$ as $n \to \infty$. To compute the expectation we use Theorem 2.7 in \cite{van2009random}, which gives
\begin{equation} \label{eq:sum_prob_indicator}
E\left[ (S_n)_p (\tilde M_n)_q \right] = \sum_{{\bf u}_1, \dots, {\bf u}_p \in \mathcal{I}} \sum_{{\bf w}_1, \dots, {\bf w}_q \in \mathcal{J}} P\left(I_{ {\bf u}_{1}} = \dots = I_{ {\bf u}_{p}} = J_{ {\bf w}_{1}} = \dots = J_{ {\bf w}_{q}} = 1\right),
\end{equation}
where the sums are taken over all the $p$-permutations, respectively $q$-permutations, of the distinct indices in $\mathcal{I}$, respectively $\mathcal{J}$. 

Next, by the fact that all stubs are uniformly paired, we have that
\begin{equation*}
P\left(I_{ {\bf u}_{1}} = \dots = I_{ {\bf u}_{p}} = J_{ {\bf w}_{1}} = \dots = J_{ {\bf w}_{q}} = 1\right),
=\frac{1}{\prod_{i=0}^{p+2q-1}(l_{n}-i)}
\end{equation*}
unless there is a conflict in the attachment rules, i.e., one stub is required to pair with two or more different stubs within the indices
$\{ {\bf u}_1,\dots,{\bf u}_{p}\}$ and $\{ {\bf w}_1,\dots,{\bf w}_{q} \}$, in which case
\begin{equation}
P\left(I_{ {\bf u}_{1}} = \dots = I_{ {\bf u}_{p}} = J_{ {\bf w}_{1}} = \dots = J_{ {\bf w}_{q}} = 1\right) = 0. \label{eq:indicator_conflict}
\end{equation}
Therefore, from \eqref{eq:sum_prob_indicator} we obtain
\begin{align}
E[(S_{n})_{p}(\tilde{M}_{n})_{q}] & \leq \sum_{{\bf u}_{1},...,{\bf u}_{p}\in\mathcal{I}} \sum_{{\bf w}_{1},...,{\bf w}_{q}\in\mathcal{J}} \frac{1}{\prod_{i=0}^{p+2q-1}(l_{n}-i)} \notag  \\
 &= \frac{|\mathcal{I}| (|\mathcal{I}| - 1) \cdots (|\mathcal{I}| - p+1)  |\mathcal{J}| (|\mathcal{J}| -1) \cdots (|\mathcal{J}| - q+1) }{l_{n}(l_{n}-1)\cdots(l_{n}-(p+2q-1))} ,
 \label{eq:ub_factorialmoment}
\end{align}
where $|A|$ denotes the cardinality of set $A$. Now note that
\begin{align*}
|\mathcal{I}| &= \sum_{i=1}^n m_{ni} d_{ni} , \qquad \text{and} \\
|\mathcal{J}| &= \sum_{1\leq i\neq j\leq n} \frac{d_{ni}(d_{ni}-1)}{2} \, m_{nj}(m_{nj}-1)  \\ 
&= \frac{1}{2} \left(\sum_{i=1}^nm_{ni}(m_{ni}-1)\right) \left(\sum_{i=1}^nd_{ni}(d_{ni}-1) \right) - \frac{1}{2} \sum_{i=1}^nm_{ni}(m_{ni}-1)d_{ni}(d_{ni}-1) .
\end{align*}
By Lemma \ref{L.MaximumGrowth} and Condition \ref{Cond.RegularityDirected} we have
$$\sum_{i=1}^nm_{ni}(m_{ni}-1)d_{ni}(d_{ni}-1) \leq \left( \max_{1\leq i\leq n} m_{ni} \right) \left( \max_{1 \leq i \leq n} d_{ni} \right) \sum_{i=1}^n m_{ni} d_{ni} =  o(n^2)$$
as $n \to \infty$. Hence, it follows from Condition \ref{Cond.RegularityDirected} that
\begin{align*}
	\frac{|\mathcal{I}|}{n} &= E[\gamma\xi]+o(1), \\
	\frac{|\mathcal{J}|}{n^2} &= \frac{1}{2} E[\gamma(\gamma-1)] E[\xi(\xi-1)] + o(1), \qquad \text{and} \\
	\frac{n}{l_n} &= \frac{1}{\mu} + o(1)
\end{align*}
as $n \to \infty$. Since $p$ and $q$ remain fixed as $n \to \infty$, we have
\begin{align*}
	\limsup_{n\rightarrow \infty}E[(S_{n})_{p}(\tilde{M}_{n})_{q}] &=\left( \lim_{n\rightarrow \infty} \frac{|\mathcal{I}|}{n} \right)^{p}
	\left( \lim_{n\rightarrow \infty}\frac{|\mathcal{J}|}{n^2} \right)^{q}
	\left( \lim_{n\rightarrow \infty}\frac{n}{l_n} \right)^{p+2q}\\
&=\left({E[\gamma\xi]}\right)^{p} \left( \frac{E[\gamma(\gamma-1)] E[\xi(\xi-1)]}{2} \right)^{q} \left( \frac{1}{\mu} \right)^{p+2q}
= \lambda_1^{p}\lambda_2^{q}.
\end{align*}
To prove the matching lower bound, we note that \eqref{eq:indicator_conflict} occurs exactly when there is a conflict in the attachment rules. Each time a conflict happens, the numerator of \eqref{eq:ub_factorialmoment} decreases by one. Therefore,
\begin{align*}
E\left[ (S_n)_p (\tilde M_n)_q \right] &= \frac{|\mathcal{I}| (|\mathcal{I}| - 1) \cdots (|\mathcal{I}| - p+1)  |\mathcal{J}| (|\mathcal{J}| -1) \cdots (|\mathcal{J}| - q+1) }{l_{n}(l_{n}-1)\cdots(l_{n}-(p+2q-1))} \\
&\hspace{5mm} - \sum_{{\bf u}_{1},...,{\bf u}_{p}\in\mathcal{I}} \sum_{{\bf w}_{1},...,{\bf w}_{q}\in\mathcal{J}} \frac{1({\bf u}_1, \dots, {\bf u}_p, {\bf w}_1, \dots, {\bf w}_q \text{ have a conflict})}{\prod_{i=0}^{p+2q-1}(l_{n}-i)} \\
&= \lambda_1^p \lambda_2^q  -  \frac{1}{(\mu n)^{p+2q}} \sum_{{\bf u}_{1},...,{\bf u}_{p}\in\mathcal{I}} \sum_{{\bf w}_{1},...,{\bf w}_{q}\in\mathcal{J}} 1({\bf u}_1, \dots, {\bf u}_p, {\bf w}_1, \dots, {\bf w}_q \text{ have a conflict}) + o(1)
\end{align*}
as $n \to \infty$. To bound the total number of conflicts note that there are three possibilities:
\begin{enumerate}
\item
a stub is assigned to two different self-loops, or
\item
a stub is assigned to a self-loop and a multiple edge, or
\item
a stub is assigned to two different multiple edges.
\end{enumerate}
We now discuss each of the cases separately. For conflicts of type (a) suppose there is a conflict between the self-loops ${\bf u}_a$ and ${\bf u}_b$; the remaining $p-2$ self-loops and $q$ pairs of multiple edges can be chosen freely. Then the number of such conflicts is bounded by $|\mathcal{I}|^{p-2} |\mathcal{J}|^q = O\left( n^{p+2q-2} \right)$, hence it suffices to show that the total number of conflicting pairs $({\bf u}_a, {\bf u}_b)$ is $o(n^2)$ as $n \to \infty$. Now, to see that this is indeed the case, first choose the node $v_i$ where the conflicting pair is; if the conflict is that an outbound stub is assigned to two different inbound stubs then we can choose the problematic outbound stub in $d_{ni}$ ways and the two inbound stubs in $m_{ni}(m_{ni}-1)$ ways, whereas if the conflict is that an inbound stub is assigned to two different outbound stubs then we can choose the problematic inbound stub in $m_{ni}$ ways and the two outbound stubs in $d_{ni} (d_{ni}-1)$ ways.  Thus, the total number of conflicting pairs is bounded by
\[
\sum_{i=1}^n(d_{ni}m_{ni}^{2}+m_{ni}d_{ni}^{2}) \leq \left(\max_{1\leq i\leq n}m_{ni} + \max_{1\leq i\leq n} d_{ni}\right) 2 \sum_{i=1}^n m_{ni} d_{ni} = o(n^{3/2}) = o(n^2).
\]

For conflicts of type (b) suppose there is a conflict between the self-loop ${\bf u}_a$ and the pair of multiple edges ${\bf w}_b$; choose the remaining $p-1$ self-loops and $q-1$ multiple edges freely. Then, the number of such conflicts is bounded by $|\mathcal{I}|^{p-1} |\mathcal{J}|^{q-1} = O\left( n^{p + 2q-3} \right)$, and it suffices to show that the number of conflicting pairs $({\bf u}_a, {\bf w}_b)$ is $o(n^3)$ as $n \to \infty$. Similarly as in case (a), an outbound stub of node $v_i$ can be paired to a self-loop and a multiple edge to node $v_j$ in $d_{ni} m_{ni} m_{nj} (d_{ni}-1) (m_{nj}-1)$ ways, and an inbound stub of node $v_i$ can be paired to a self-loop and a multiple edge from node $v_j$ in $m_{ni} d_{ni} d_{nj} (m_{ni}-1) (d_{nj}-1)$ ways, and so the total number of conflicting pairs is bounded by
\begin{align*}
 \sum_{i=1}^n\sum_{j=1}^n(d_{ni}^{2}m_{ni}m_{nj}^{2}+m_{ni}^{2}d_{ni}d_{nj}^{2})\leq 
 \left(\max_{1\leq i\leq n}m_{ni}+\max_{1\leq i\leq n}d_{ni}  \right) 2 \left(\sum_{i=1}^n m_{ni}^2 \right) \left(\sum_{i=1}^n d_{ni}^2 \right)
=o(n^{5/2}) = o(n^3).
\end{align*}

Finally, for conflicts of type (c) we first fix ${\bf w}_a$ and ${\bf w}_b$ and choose freely the remaining $p$ self-loops and $q-2$ multiple edges, which can be done in less than $|\mathcal{I}|^p |\mathcal{J}|^{q-2} = O\left( n^{p+ 2q - 4} \right)$ ways. It then suffices to show that the number of conflicting pairs $({\bf w}_a, {\bf w}_b)$ is $o(n^4)$ as $n \to \infty$. A similar reasoning to that used in the previous cases gives that 
the total number of conflicting pairs is bounded by
\begin{align*}
&2\sum_{i=1}^n\sum_{j=1}^n\sum_{k=1}^n(d_{ni}^{3}m_{nj}^{2}m_{nk}^{2}+m_{ni}^{3}d_{nj}^{2}d_{nk}^{2}) \\
&\leq 2\left(\max_{1\leq i\leq n}m_{ni}+\max_{1\leq i\leq n}d_{ni}  \right)
\left(\sum_{i=1}^nd_{ni}^2\left(\sum_{i=1}^nm_{ni}^2\right)^2+\sum_{i=1}^nm_{ni}^2\left(\sum_{i=1}^nd_{ni}^2\right)^2\right)  \\
&=o(n^{7/2}) = o(n^4).
\end{align*}

We conclude that in any of the three cases the number of conflicts is negligible, which completes the proof.
\end{proof}

\begin{proof}[Proof of Proposition \ref{P.RepeatedDistr}]
Let $\mathcal{S}_n$ be the event that the resulting graph is simple, and note that the bi-degree-sequence $({\bf N}^{(r)}, {\bf D}^{(r)})$ is the same as $({\bf N}, {\bf D})$ given $\mathcal{S}_n$. 

To prove part (a) note that for any $i,j = 0, 1, 2, \dots$,
\begin{align*}
h^{(n)}(i,j) &= \frac{1}{n} \sum_{i=1}^n P( N_k = i, D_k = j | \mathcal{S}_n)  = \frac{1}{P(\mathcal{S}_n)} P(N_1 = i, D_1 = j, \mathcal{S}_n),   
\end{align*}
since the $\{ (N_k, D_k)\}_{k=1}^n$ are identically distributed. Now let $\mathcal{G}_n = \sigma(N_1, \dots, N_n, D_1, \dots, D_n)$ and condition on $\mathcal{G}_n$ to obtain
$$P( N_1 = i, D_1 = j, \mathcal{S}_n) = E[ 1(N_1 = i, D_1 = j) P(\mathcal{S}_n | \mathcal{G}_n) ], $$
from where it follows that
\begin{align*}
\left| h^{(n)}(i,j) - f_i g_j \right| &\leq \left|  \frac{E[ 1(N_1 = i, D_1 = j) (P(\mathcal{S}_n | \mathcal{G}_n) - P(\mathcal{S}_n) ) ]}{P(\mathcal{S}_n)}   \right| + \left| P(N_1 = i, D_1 = j) - f_i g_j  \right| \\
&\leq  E\left[  \left| \frac{P(\mathcal{S}_n | \mathcal{G}_n)}{P(\mathcal{S}_n)} - 1 \right| \right]  + \left| P( N_1 = i, D_1 = j) - f_i g_j \right|.
\end{align*}
Theorem \ref{T.WeakConvergence} gives that the second term converges to zero, and for the first term use Theorem \ref{T.ProbSimple} to obtain that
both $P(\mathcal{S}_n)$ and $P(\mathcal{S}_n | \mathcal{G}_n)$ converge to the same positive limit, so by dominated convergence, 
$$\lim_{n \to \infty} E\left[  \left| \frac{P(\mathcal{S}_n | \mathcal{G}_n)}{P(\mathcal{S}_n)} - 1 \right| \right] \leq E\left[  \lim_{n \to \infty} \left| \frac{P(\mathcal{S}_n | \mathcal{G}_n)}{P(\mathcal{S}_n)} - 1 \right| \right] = 0.$$

For part (b) we only show the proof for $\widehat{g_k}^{(n)}$ since the proof for $\widehat{f_k}^{(n)}$ is symmetrical. Note that $\widehat{g_k}^{(n)}$ is a quantity defined on $\mathcal{S}_n$. Fix $\epsilon > 0$ and use the union bound to obtain
\begin{align}
P\left( \left. \left| \widehat{g_k}^{(n)} - g_k \right| > \epsilon \right| \mathcal{S}_n \right) &\leq \frac{1}{P(\mathcal{S}_n)} P\left(  \left| \frac{1}{n} \sum_{i=1}^n 1(D_i = k) - g_k \right| > \epsilon  \right) \notag \\
&\leq  \frac{1}{P(\mathcal{S}_n)} P\left( \left. \frac{1}{n} \sum_{i=1}^n \left| 1(\xi_i + \chi_1 = k) - 1(\xi_i = k) \right|  > \epsilon/2 \right| \mathcal{D}_n \right) \label{eq:alreadydone} \\
&\hspace{5mm} +  \frac{1}{P(\mathcal{S}_n)P(\mathcal{D}_n)} P \left( \left| \frac{1}{n} \sum_{i=1}^n 1(\xi_i = k) - g_k \right| > \epsilon/2   \right). \label{eq:ChebyshevAgain}
\end{align}
By Theorem \ref{T.ProbSimple} and Lemma \ref{L.DifferenceLimit}, $P(\mathcal{S}_n)$ and $P(\mathcal{D}_n)$ are bounded away from zero, so we only need to show that the numerators converge to zero. The arguments are the same as those used in the proof of Proposition~\ref{P.MainConditions}; for \eqref{eq:ChebyshevAgain} use Chebyshev's inequality to obtain that
$$P \left( \left| \frac{1}{n} \sum_{i=1}^n 1(\xi_i = k) - g_k \right| > \epsilon/2   \right) \leq \frac{\var(1(\xi_1 = k))}{n(\epsilon/2)^2} = O( n^{-1} ),$$
as $n \to \infty$, and for \eqref{eq:alreadydone} 
\begin{align*}
P\left( \left. \frac{1}{n} \sum_{i=1}^n \left| 1(\xi_i + \chi_i = k) - 1(\xi_i = k) \right|  > \epsilon/2 \right| \mathcal{D}_n \right) 
&\leq P\left( \left. \frac{1}{n} \sum_{i=1}^n 1(\chi_i = 1) > \epsilon/2 \right| \mathcal{D}_n \right) \\
&\leq P\left( \left. \frac{|\Delta_n|}{n} > \epsilon/2 \right| \mathcal{D}_n \right) \leq 1(n^{-\kappa+\delta_0}  > \epsilon/2),
\end{align*}
which also converges to zero. This completes the proof. 
\end{proof}

Finally, the last result of the paper, which refers to the erased directed configuration model, is given below. Since the technical part of the proof is to show that the probability that no in-degrees or out-degrees of a fixed node are removed during the erasing procedure, we split the proof of Proposition \ref{P.ErasedDistr} into two parts. The following lemma contains the more delicate step.

\begin{lemma} \label{L.SurvivingStubs}
Consider the graph obtained through the erased directed configuration model using as bi-degree-sequence $({\bf N}, {\bf D})$, as constructed in Subsection \ref{SS.Algorithm}. Let $E^+$ and $E^-$ be the number of inbound stubs and outbound stubs, respectively, that have been removed from node $v_1$ during the erasing procedure. Then,
$$\lim_{n \to \infty} P(E^+ = 0) = 1 \qquad \text{and} \qquad \lim_{n\to\infty} P(E^- = 0) = 1.$$
\end{lemma}

\begin{proof}
We only show the result for $E^+$ since the proof for $E^-$ is symmetric. Define the set
\begin{align*}
&\mathcal{P}^+_n = \{ (i_1, \dots, i_t): 2 \leq i_1 \neq i_2 \cdots \neq i_t \leq n, \, 1 \leq t \leq n\},
\end{align*}
and note that in order for all the inbound stubs of node $v_1$ to survive the erasing procedure, it must have been that they were paired to outbound stubs of $N_1$ different nodes from $\{v_2, \dots, v_n\}$. Before we proceed it is helpful to recall some definitions from Section \ref{S.DegreeSequences}, $L_n = \sum_{i=1}^n N_i = \sum_{i=1}^n D_i $, $\Gamma_n = \sum_{i=1}^n \gamma_i$, $\Xi_n = \sum_{i=1}^n \xi_i$, $\Delta_n = \Gamma_n - \Xi_n$, and $\mathcal{D}_n = \{ |\Delta_n| \leq n^s\}$, where $s = 1-\kappa + \delta_0$; also, $\{\gamma_i\}$ and $\{\xi_i\}$ are independent sequences of i.i.d. random variables having distributions $F$ and $G$, respectively.  Now fix $0 < \epsilon < 1-s$ and let $\mathcal{G}_n = \sigma(N_1, \dots, N_n, D_1, \dots, D_n)$. Then, since $D_i = \xi_i + \chi_i \geq \xi_i$, 
\begin{align}
P\left(E^+ = 0 \right) &= E\left[ \left. P\left( E^+ = 0  \right| \mathcal{G}_n \right) \right] \geq E\left[ \left. P\left( E^+ = 0  \right| \mathcal{G}_n \right) 1(1\leq N_1 \leq n^\epsilon) \right] + P(N_1 = 0) \notag \\
&= E\left[  \frac{1(1 \leq N_1 \leq n^\epsilon) }{L_n!} \sum_{(i_1, i_2, \dots, i_{N_1}) \in \mathcal{P}^+_n} D_{i_1} D_{i_2} \cdots D_{i_{N_1}}   (L_n - N_1)!   \right]  + P(N_1 = 0) \notag \\
&\geq  E\left[  \left. \frac{1(1\leq \gamma_1+\tau_1 \leq n^\epsilon) }{L_n!} \sum_{(i_1, i_2, \dots, i_{(\gamma_1+\tau_1)}) \in \mathcal{P}^+_n} \xi_{i_1} \xi_{i_2} \cdots \xi_{i_{(\gamma_1+\tau_1)} }   (L_n - \gamma_1 - \tau_1)!  \right| \mathcal{D}_n \right]  \notag \\
&\hspace{5mm} + P(N_1 = 0) \notag \\
&\geq E\left[ \left.  \frac{1(1 \leq \gamma_1 \leq n^\epsilon) 1(\tau_1 = 0)}{(L_n)^{\gamma_1}} \sum_{(i_1, i_2, \dots, i_{\gamma_1}) \in \mathcal{P}^+_n} \xi_{i_1} \xi_{i_2} \cdots \xi_{i_{\xi_1} }  \right| \mathcal{D}_n   \right] + P(N_1 = 0). \label{eq:OutboundSurvive}
\end{align}
Next, condition on $\mathcal{F}_n = \sigma(\gamma_1, \dots, \gamma_n, \xi_1, \dots, \xi_n)$ and note that
$$P(\tau_1 = 0 | \mathcal{F}_n) = 1\left( \Delta_n \geq 0 \right)+ \frac{\Gamma_n}{\Gamma_n+ |\Delta_n|} 1(\Delta_n < 0) \geq  \frac{\Gamma_n}{\Gamma_n+|\Delta_n|}.$$
It follows that the expectation in \eqref{eq:OutboundSurvive} is equal to
\begin{align*}
&E\left[ \left. P(\tau_1 = 0 | \mathcal{F}_n) \frac{1(1\leq \gamma_1 \leq n^\epsilon) }{(L_n)^{\gamma_1}} \sum_{(i_1, i_2, \dots, i_{\gamma_1}) \in \mathcal{P}^+_n} \xi_{i_1} \xi_{i_2} \cdots \xi_{i_{\gamma_1} } \right| \mathcal{D}_n   \right] \\
&\geq E\left[ \left. \frac{\Gamma_n}{\Gamma_n+|\Delta_n|} \cdot \frac{1(1 \leq \gamma_1 \leq n^\epsilon) }{(\Gamma_n + |\Delta_n|)^{\gamma_1}} \sum_{(i_1, i_2, \dots, i_{\gamma_1}) \in \mathcal{P}^+_n} \xi_{i_1} \xi_{i_2} \cdots \xi_{i_{\gamma_1} } \right| \mathcal{D}_n    \right] \\
&\geq E\left[ \left.   \frac{1(1 \leq \gamma_1 \leq n^\epsilon) \Gamma_n}{(\Gamma_n + n^s)^{\gamma_1+1}} \sum_{(i_1, i_2, \dots, i_{\gamma_1}) \in \mathcal{P}^+_n} \xi_{i_1} \xi_{i_2} \cdots \xi_{i_{\gamma_1} } \right| \mathcal{D}_n    \right] 
\end{align*}
\begin{align*}
&= \frac{1}{P(\mathcal{D}_n)} E\left[ 1(1 \leq \gamma_1 \leq n^\epsilon) \sum_{(i_1, i_2, \dots, i_{\gamma_1}) \in \mathcal{P}^+_n} E\left[ \left. \frac{ 1(\mathcal{D}_n) \Gamma_n}{(\Gamma_n + n^s)^{\gamma_1+1}} \cdot \xi_{i_1} \xi_{i_2} \cdots \xi_{i_{\gamma_1} } \right| \gamma_1 \right]   \right] \\
&= \frac{1}{P(\mathcal{D}_n)} E\left[ 1(1 \leq \gamma_1 \leq n^\epsilon) \frac{(n-1)!}{(n-1-\gamma_1)! n^{\gamma_1}}  E\left[ \left.  \frac{ 1(\mathcal{D}_n) \Gamma_n n^{\gamma_1} }{(\Gamma_n + n^s)^{\gamma_1+1}} \cdot \xi_{1} \xi_{2} \cdots \xi_{\gamma_1} \right| \gamma_1 \right]   \right].
\end{align*}

It follows by Fatou's lemma, Lemma \ref{L.DifferenceLimit} and Theorem \ref{T.WeakConvergence} that
\begin{align*}
\liminf_{n \to \infty} P(E^+ = 0) &\geq  E\left[ 1(\gamma_1 \geq 1) \liminf_{n \to \infty} E\left[ \left.  \frac{ 1(\mathcal{D}_n) \Gamma_n n^{\gamma_1} }{(\Gamma_n + n^s)^{\gamma_1+1}} \cdot \xi_{1} \xi_{2} \cdots \xi_{\gamma_1} \right| \gamma_1 \right] \right] + P(\gamma_1 = 0).
\end{align*}
Next, define the function $u_n^+: \mathbb{N} \to [0, \infty)$ as
\begin{align*}
u_n^+(t) &= E\left[  \frac{1( |\Gamma_{n-1} + t - \Xi_{n} | \leq n^s) (\Gamma_{n-1} + t) n^t}{(\Gamma_{n-1} + t + n^s)^{t+1}} \cdot \xi_1 \xi_2 \cdots \xi_t \right], \\
\end{align*}
and note that it only remains to prove that for all $t \in \mathbb{N}$, $\liminf_{n \to \infty} u_n^+(t) = 1$.
 
Now let $0 < a < \mu$ and note that
\begin{align*}
u_n^+(t) &\geq E\left[  \frac{1( |\Gamma_{n-1} + t - \Xi_{n} | \leq n^s)  }{\mu^t } \cdot \xi_1 \xi_2 \cdots \xi_t \right] - P(\Gamma_{n-1} < an) \\
&\hspace{5mm} - E\left[  1(\Gamma_{n-1} \geq an) \left| \frac{(\Gamma_{n-1}+t) n^t}{(\Gamma_{n-1}+t + n^s)^{t+1}} - \frac{1}{\mu^t} \right| \xi_1 \xi_2 \cdots \xi_t \right].  
\end{align*} 
The SLLN and bounded convergence give $\lim_{n \to \infty} P(\Gamma_{n-1} < an) = 0$ and 
\begin{align*}
&\limsup_{n \to \infty}  E\left[  1(\Gamma_{n-1} \geq an) \left| \frac{(\Gamma_{n-1}+t) n^t}{(\Gamma_{n-1}+t + n^s)^{t+1}} - \frac{1}{\mu^t} \right| \xi_1 \xi_2 \cdots \xi_t \right] \\
&\leq E\left[ \xi_1 \xi_2 \cdots \xi_t \limsup_{n \to \infty}  \left| \frac{(\Gamma_{n-1}+t) n^t}{(\Gamma_{n-1}+t + n^s)^{t+1}} - \frac{1}{\mu^t} \right|  \right] = 0,
\end{align*}
from where it follows that
$$\liminf_{n \to \infty} u_{n}^+(t) \geq \liminf_{n\to \infty} E\left[  \frac{1(|\Gamma_{n-1} + t - \Xi_{n} | \leq n^s) }{\mu^t } \cdot \xi_1 \xi_2 \cdots \xi_t \right] . $$
The last step is to condition on $\xi_1, \xi_2 \dots, \xi_t$ and use Fatou's Lemma again to obtain
\begin{align*}
&\liminf_{n\to \infty} E\left[  \frac{1(|\Gamma_{n-1}+t - \Xi_{n} | \leq n^s) }{\mu^t } \cdot \xi_1 \xi_2 \cdots \xi_t \right]  \\
&= \liminf_{n\to \infty} E\left[  \frac{ \xi_1 \xi_2 \cdots \xi_t   }{\mu^t }  P(|\Gamma_{n-1} + t - \Xi_{n}  | \leq n^s | \xi_1, \dots, \xi_t) \right] \\
&\geq  E\left[  \frac{ \xi_1 \xi_2 \cdots \xi_t   }{\mu^t } \liminf_{n\to \infty} P( |\Gamma_{n-1} + t - \Xi_{n}| \leq n^s | \xi_1, \dots, \xi_t) \right] .
\end{align*}
Finally, by the same reasoning used in the proof of Lemma \ref{L.DifferenceLimit}, we obtain
$$\lim_{n \to \infty} P( |\Gamma_{n-1} + t - \Xi_{n}| \leq n^s | \xi_1, \dots, \xi_t) =  1 \qquad \text{a.s.}$$
Since $E[\xi_1 \xi_2 \cdots \xi_t] /\mu ^t = 1$, this completes the proof. 
\end{proof}

\begin{proof}[Proof of Proposition \ref{P.ErasedDistr}]
To prove part (a) note that since the $\{ (N_i^{(e)}, D_i^{(e)}) \}_{i=1}^n$ are identically distributed, then $h^{(n)}(i,j) = P(N_1^{(e)} = i, D_1^{(e)} = j)$. It follows that
\begin{align*}
\left| h^{(n)}(i,j) - f_i g_j \right| &\leq \left|  P(N_1^{(e)} = i, D_1^{(e)} = j) - P(N_1 = i, D_1 = j)  \right| + \left| P(N_1 = i, D_1 = j) - f_i g_j \right| .
\end{align*}
By Theorem \ref{T.WeakConvergence} we have that $\left| P(N_1 = i, D_1 = j) - f_i g_j \right| \to 0$, as $n \to \infty$, and for the remaining term note that
\begin{align}
&\left|  P(N_1^{(e)} = i, D_1^{(e)} = j) - P(N_1 = i, D_1 = j)  \right| \notag \\
&\leq E\left[ \left| 1(N_1^{(e)} = i, D_1^{(e)} = j) - 1(N_1 = i, D_1 = j) \right| \right] \notag \\
&\leq E\left[ \left| 1(D_1^{(e)} = j) - 1(D_1 = j) \right| \right]  + E \left[ \left| 1(N_1^{(e)} = i) - 1(N_1 = i) \right| \right] . \label{eq:ErasedDifferences} 
\end{align}
To bound the expectations in \eqref{eq:ErasedDifferences} let $E^+$ and $E^-$ be the number of  inbound stubs and outbound stubs, respectively, that have been removed from node $v_1$ during the erasing procedure. Then,
\begin{align*}
E\left[ \left| 1(D_1^{(e)} = j) - 1(D_1 = j) \right| \right] &\leq P\left( E^- \geq 1 \right) \quad \text{and} \\
E \left[ \left| 1(N_1^{(e)} = i) - 1(N_1 = i) \right| \right] &\leq P\left( E^+ \geq 1 \right).
\end{align*}
By Lemma \ref{L.SurvivingStubs}, 
$$\lim_{n \to \infty} P(E^- \geq 1) = 0 \qquad \text{and} \qquad \lim_{n \to \infty} P(E^+ \geq 1) = 0,$$
which completes the proof of part (a). 

For part (b) we only show the proof for $\widehat{g_k}^{(n)}$, since the proof for $\widehat{f_k}^{(n)}$ is symmetrical. Fix $\epsilon > 0$ and use the triangle inequality and the union bound to obtain
\begin{align*}
P\left( \left| \widehat{g_k}(k) - g_k \right| > \epsilon \right) &\leq P\left( \left| \widehat{g_k}(k) - \frac{1}{n} \sum_{i=1}^n 1(D_i = k) \right| > \epsilon/2 \right) + P\left(  \left| \frac{1}{n} \sum_{i=1}^n 1(D_i = k) - g_k \right| > \epsilon/2 \right).
\end{align*}
From the proof of Proposition \ref{P.RepeatedDistr}, we know that the second probability converges to zero as $n \to \infty$, and for the first one use Markov's inequality to obtain
\begin{align*}
P\left( \left| \widehat{g_k}(k) - \frac{1}{n} \sum_{i=1}^n 1(D_i = k) \right| > \epsilon/2 \right) &\leq P\left( \frac{1}{n} \sum_{i=1}^n \left| 1(D_i^{(e)} = k) - 1(D_i = k) \right| > \epsilon/2 \right) \\
&\leq \frac{2}{\epsilon} E\left[ \left| 1(D_1^{(e)} = k) - 1(D_1 = k) \right| \right] \\
&\leq \frac{2}{\epsilon} P(E^- \geq 1) \to 0,
\end{align*}
as $n \to \infty$, by Lemma \ref{L.SurvivingStubs}. 
\end{proof}

\bibliographystyle{plain}

\end{document}